\documentclass{article}

\usepackage{microtype}
\usepackage{graphicx}
\usepackage{subfigure}
\usepackage{booktabs} % for professional tables
\usepackage{amsfonts} % mathfrak
\usepackage{bbm}
\usepackage{xcolor}
\usepackage{amsmath}
\usepackage{authblk}

\usepackage{hyperref}
\usepackage[inline]{enumitem}

\usepackage[accepted]{icml2020}

\usepackage{mathtools}
\usepackage{graphicx} % rescale equation size

\usepackage{amsmath, amsthm} %for theorems

\usepackage{tikz}
\usetikzlibrary{automata, positioning, shapes.geometric}

\newcommand{\argmax}{\mathop{\mathrm{arg\,max}}}

\DeclareMathOperator{\E}{\mathbb{E}}
\DeclareMathOperator{\St}{\mathcal{S}}
\DeclareMathOperator{\A}{\mathcal{A}}
\DeclareMathOperator{\Prec}{\mathcal{P}}
\DeclareMathOperator{\w}{\mathbf{w}}

\DeclarePairedDelimiter\abs{\lvert}{\rvert}% | |
\DeclarePairedDelimiter\norm{\lVert}{\rVert}% || ||

\DeclarePairedDelimiter\innorm{\langle}{\rangle}% < >

 \newcommand{\ie}{i.\,e., }
 \newcommand{\eg}{e.\,g., }

\newenvironment{Proof}[1][Proof]
  {\proof[#1]\leftskip=1cm\rightskip=1cm}
  {\endproof}

\icmltitlerunning{Distributional Robustness and Regularization in Reinforcement Learning}

\begin{document}

\newtheorem{lemma}{Lemma}[section] 
\newtheorem*{lemma*}{Lemma} 
\newtheorem{proposition}{Proposition}[section] 
\newtheorem*{proposition*}{Proposition}
\newtheorem{corollary}{Corollary}[section]
\newtheorem{remark}{Remark}[section]
\newtheorem{assumption}{Assumption}[section]
\newtheorem{definition}{Definition}[section]
\newtheorem{example}{Example}[section]
\newtheorem{theorem}{Theorem}[section]
\newtheorem*{theorem*}{Theorem}

\twocolumn[
\icmltitle{Distributional Robustness and Regularization in Reinforcement Learning}  




\begin{icmlauthorlist}
\icmlauthor{Esther Derman}{tech}
\icmlauthor{Shie Mannor}{tech}
\end{icmlauthorlist}

\icmlaffiliation{tech}{Technion, Israel}

\icmlcorrespondingauthor{Esther Derman}{estherderman@campus.technion.ac.il}
\icmlcorrespondingauthor{Shie Mannor}{shie@ee.technion.ac.il}

\icmlkeywords{Machine Learning, ICML}

\vskip 0.3in]



\printAffiliationsAndNotice{}  % leave blank if no need to mention equal contribution

\begin{abstract}
Distributionally Robust Optimization (DRO) has enabled to prove the equivalence between robustness and regularization in classification and regression, thus providing an analytical reason why regularization generalizes well in statistical learning. 
Although DRO's extension to sequential decision-making overcomes \emph{external uncertainty} through the robust Markov Decision Process (MDP) setting, the resulting formulation is hard to solve, especially on large domains. 
On the other hand, existing regularization methods in reinforcement learning (RL) only address \emph{internal uncertainty} due to stochasticity. 
Our study aims to facilitate robust RL by establishing a dual relation between robust MDPs and regularization. We introduce Wasserstein distributionally robust MDPs and prove that they hold out-of-sample performance guarantees. We also define a new regularizer on empirical value functions and show that it lower bounds the Wasserstein distributionally robust value function. Then, we extend the result to linear value function approximation for large state spaces. Our approach provides an alternative formulation of robustness with guaranteed finite-sample performance. 
Moreover, it suggests using our regularizer as a practical tool for dealing with $\emph{external uncertainty}$ in RL. 
\end{abstract}

\section{Introduction}
\label{section:introduction}

Markov Decision Processes (MDPs) originated from the seminal works of \citet{bellman1957markovian} and \citet{howard1960dynamic} to model sequential decision-making problems and provide a theoretical basis for RL methods. Real-world applications which include healthcare and marketing, for example, give rise to several challenging issues. Firstly, the model parameters are generally unknown but rather estimated through historical data. This may lead the performance of a learned strategy to significantly degrade when deployed \cite{mannor2007bias}. Secondly, experimentation can be expensive or time-consuming which constrains policy evaluation and improvement to perform with limited data \cite{lange2012batch}. Lastly, when the state-space is large, the value function is commonly approximated by a parametric function, which results in additional uncertainty regarding the efficiency of a learned policy \cite{farahmand2009regularized, farahmand2011regularization}.

This phenomenon is reminiscent of over-fitting in statistical learning that can be interpreted as the following single-stage decision-making problem \cite{zhang2018study}. Consider a training set of random input-output vectors $(\widehat{x}_i, \widehat{y}_i)_{i=1}^{n}$ generated by a fixed distribution and assume one wants to find a parameter $\theta\in\Theta$ that minimizes the expected loss function $\ell_{\theta}$ with respect to (w.r.t.) the generating distribution. In general, the true distribution is unknown and hard to estimate accurately. A classical method to overcome this is to minimize the empirical risk: $\min_{\theta\in\Theta}\frac{1}{n}\sum_{i=1}^n \ell_\theta(\widehat{x}_i, \widehat{y}_i)$, but this often yields solutions that perform poorly on out-of-sample data \cite{friedman2001elements}. 

Several methods ensure better generalization to new, unseen data (test set) while performing well on available data (training set). These may be categorized into two main approaches. The first one \emph{regularizes} the empirical risk and optimizes the resulting objective \cite{vapnik2013nature}. 
Another approach \emph{robustifies} the objective function by introducing ambiguity w.r.t. the empirical distribution \cite{kuhn2019wasserstein}. The resulting problem can be formulated as 
% \begin{equation}
% \label{eq:drso}
  $\text{(DRO)}: \min_{\theta\in\Theta}\sup_{\mathbb{Q}\in \mathfrak{M}(\widehat{\mathbb{P}}_n)} \mathbb{E}_{(x,y)\sim\mathbb{Q}}[\ell_{\theta}(x,y)],$ 
%   \tag{DRSO}
% \end{equation}
where $\widehat{\mathbb{P}}_n$ is the empirical distribution w.r.t. the sample set and $\mathfrak{M}(\widehat{\mathbb{P}}_n)$ is an ambiguity set of probability distributions consistent with the dataset. Such ambiguity sets can be based on \emph{specified properties} such as moment constraints \cite{data_driven, partial_lifting, wiesemann2014distributionally}, or on a given \emph{divergence} from the empirical distribution \cite{hu2013kullback,phi_div,erdougan2006ambiguous,DRO_kuhn}. The resulting problem can be solved using Distributionally Robust Optimization (DRO). 

Wasserstein distance-based ambiguity sets are of particular interest in DRO theory, as they display interesting ramifications in main statistical learning problems. More precisely, the specific problem (DRO) is equivalent to regularization for fundamental learning tasks such as classification \cite{xu2009robustness, abadeh_regression, blanchet_equivalence}, regression \cite{abadeh_mass} and maximum likelihood estimation \cite{kuhn2019wasserstein}. However, 
% to the best of our knowledge, 
equivalence between robustness and regularizaion has only been studied on single-stage decision problems. 

Regularization techniques are widely used in RL to mitigate uncertainty in value function approximation \cite{farahmand2009regularized} or to derive improved versions of policy optimization methods \cite{shani2019adaptive}. Although regularized policy learning helps to derive risk-sensitive strategies that satisfy safety criteria \cite{ruszczynski2010risk, tamar2015policy}, existing connections between regularization in RL and robustness are still weak: Prior regularization methods address the \emph{internal uncertainty} \ie the inherent stochasticity of the dynamical system, without accounting for the \emph{external uncertainty} of the MDP \ie transition and reward functions. 
%Distributionally Robust MDPs (DRMDPs) including the special case of 
Although robust MDPs provide a convenient framework for dealing with \emph{external uncertainty} and enabling
% to construct robust strategies with 
better generalization in sequential decision-making  \cite{iyengar2005robust, nilim2005robust,xu2010distributionally, yu2015distributionally}, 
solving them remains challenging even on small domains, mainly because it is hard to construct an uncertainty set that yields a robust policy without being too conservative \cite{petrik2019beyond}. 

Our study aims to facilitate robust RL by addressing a new regularization perspective on sequential decision-making settings. In Sec. \ref{section:background}, we recall the MDP framework and describe its robust and distributionally robust formulations. In Sec. \ref{section:wdrmdp}, we introduce Wasserstein distributionally robust MDPs as an analytical tool to establish a connection between robustness and regularization.
We address our main result in Sec. \ref{section:reg_WDRMDP}: In Thm. \ref{theorem:tabular_equivalence}, we devise the first dual relation between robustness to model uncertainty and regularized value functions. An extension to linear function approximation is addressed and formally stated in Thm.~\ref{theorem:approximate_equivalence}. Finally, we establish out-of-sample guarantees for Wasserstein distributionally robust MDPs, thus demonstrating the fact that our regularization method enables better generalization to unseen data. All proofs can be found in the Appendix.

\textbf{Related Work. }
Regularization in statistical learning precedes its robust formulations. Indeed, a robust optimization interpretation has first been suggested by \citet{xu2009robustness} for support vector machines, long after the regularization methods of \citet{vapnik2013nature}. Then, advancing research on data-driven DRO has enabled to establish equivalence between robustness and regularization in a wider range of statistical learning problems \cite{abadeh_regression, abadeh_mass,blanchet_equivalence, kuhn2019wasserstein}. Differently, in RL, RMDPs date back to \citeyear{iyengar2005robust} with the concurrent works \cite{iyengar2005robust,nilim2005robust} and their extension to DRMDPs \cite{xu2010distributionally, yu2015distributionally, yang2017convex,yang2018wasserstein, chen2019distributionally}, while to our knowledge, our study suggests the first connection between regularization and robustness to parameter uncertainty in RL.  

Moreover, distributional RL as in \cite{bellemare2017distributional} differs from our approach. There, an optimal policy is learned through the \emph{internal distribution} of the cumulative reward while we study its \emph{worst-case expectation} to account for the \emph{external uncertainty} of the MDP parameters. Also, \citet{bellemare2017distributional} consider a different metric, namely a minimum of Wasserstein distances over the state-action space. Such a metric is problematic in our setting, as it can falter the rectangularity assumption (see Sec.
~\ref{section:rmdp}-\ref{section:drmdp}). 

\textbf{Notation. }
% We denote by $\overline{\mathbb{R}} = [-\infty, +\infty]$ the extended reals. 
$\mathcal{M}(\mathcal{E})$ denotes the set of distributions over a Borel set $\mathcal{E}$. 
% Given a norm $\norm{\cdot}$ over an Euclidean space, the dual norm is defined through $\norm{\cdot}_* = \sup_{\norm{x}\leq 1}\innorm{\cdot, x}$. 
For all $n\in\mathbb{N}$, we define $[n]:=\{1,\cdots,n\}$.

\section{From MDPs to distributionally robust MDPs}
\label{section:background}
This section provides the theoretical background used throughout this study:  
% We first recall some definitions and fundamental results of convex analysis. This preliminary study is  further applied in Section \ref{section:reg_WDRMDP}, where we define the \textit{conjugate robust value function}, a tool borrowed from the conjugate transformation, which in turn plays a crucial role in our regularization method. Then, 
It describes the MDP setting and its generalization to robust and distributionally robust MDPs.  

\subsection{Markov Decision Process}
A Markov Decision Process (MDP) is a tuple $\langle \St, \A, r,  p\rangle$ with finite state and action spaces $\St$ and $\A$ respectively,
such that $r: \St\times \A \rightarrow \mathbb{R}$ is a deterministic reward function bounded by $R_{\max}$ and  $p: \St  \rightarrow \mathcal{M}(\St)^{\abs{\A}}$ denotes the transition model \ie for all $s\in\St$, the elements of $p_s := (p(\cdot |s,a_1),\cdots, p(\cdot |s,a_{\abs{A}}))\in\mathcal{M}(\St)^{\abs{\A}} \subset \mathbb{R}^{\abs{\St}\times\abs{\A}}$ are listed in such a way that transition probabilities of the same action are arranged in the same block. At step $t$, the agent is in state $s_t$, chooses action $a_t$ according to a policy $\pi: \St \rightarrow \mathcal{\A}$ and gets a reward $r(s_t,a_t)$. It is then brought to state $s_{t+1}$ with probability $p(s_{t+1}| s_t, a_t)$. 

The agent's goal is to maximize the following \emph{value function} over the set of policies $\Pi$ for all $s\in\St$:
% \begin{equation}
%     \label{eq: value_function}
   $ v^\pi_p(s) = \mathbb{E}_p^\pi\left[\sum_{t=0}^\infty \gamma^t r(s_t,a_t)\biggm{|}  s_0=s\right]$, 
% \end{equation}
where $\gamma\in [0,1)$ is a discount factor 
% determining the degree of myopia to upcoming rewards 
and the expectation is conditioned on transition model $p$, policy $\pi$ and initial state $s$. It can be efficiently computed thanks to the \emph{Bellman operator} contraction, which admits  
% \begin{align*}
% T_p^\pi v(s) = \sum_{a\in\A}\pi_s(a)\left(r(s,a) + \gamma \sum_{s'\in\St}p(s'|s,a) v(s')\right) 
% \end{align*}
% defined for $s\in\St$. Besides being non-decreasing, $T_{p}^\pi$ is a $\gamma$-contraction w.r.t. the sup-norm so it admits 
$v^\pi_p$ as a unique fixed point \cite{puterman2014markov}.

\subsection{Robust Markov Decision Process}
\label{section:rmdp}

A robust MDP $\langle \St, \A, r,  \Prec\rangle$ is an MDP with uncertain transition model $p\in\Prec$. We assume that the \emph{uncertainty set} $\Prec$ 
% to be structured as a product set of transition models that are independent for each state, \ie $\Prec \subset \mathbb{R}^{\abs{\St}\times\abs{\A}\times\abs{\St}}$
is $(s,a)$-rectangular, \ie $\Prec =\bigotimes_{s\in\St} \Prec_s = \bigotimes_{s\in\St, a\in\A}\Prec_{s,a},$ where for all $s\in\St$, $\Prec_{s}$ is a set of transition matrices $p_s \in\Prec_s$ \cite{wiesemann2013robust}. Accordingly, for all $s,s'\in\St$ and $a\in\A$, the probability of getting from state $s$ to state $s'$ after applying action $a$ is given by any $p_s\in\Prec_s$. Moreover, we assume that $\Prec$ is closed, convex and compact. 

The \emph{robust value function} under any policy $\pi$ is the worst-case performance:
$v_{\Prec}^\pi(s) = \inf_{p\in\Prec}v_p^\pi(s), \quad \forall s\in\St$, and a policy is robust optimal whenever it solves the max-min problem $\max_{\pi\in\Pi}\min_{p\in\Prec}v_p^\pi$.
Thanks to the rectangularity assumption, one can show that $v_{\Prec}^\pi$ is the unique fixed point of a contracting \emph{robust Bellman operator} \cite{iyengar2005robust, nilim2005robust}, and a robust MDP can be solved efficiently using robust dynamic programming.

\subsection{Distributionally Robust Markov Decision Process}
\label{section:drmdp}
In a Distributionally Robust Markov Decision Process (DRMDP) $\langle \St, \A, r,  \Prec, \mathcal{M}\rangle$, the transition model is also unknown but instead, it is a random variable supported on $\Prec$ and obeying a distribution $\mu\in\mathcal{M}\subseteq \mathcal{M}(\Prec)$ \cite{xu2010distributionally,yu2015distributionally}. Here, the class of probability distributions $\mathcal{M}$ is the \emph{ambiguity set} and each of them is supported on the uncertainty set $\Prec$. 
% Note that DRMDPs generalize robust MDPs: if $\mathcal{M}$ only contains Dirac distributions $\delta_{p}$-s \ie with full mass on transition $p$, then we recover a robust MDP of uncertainty set $\Prec := \{p: \delta_{p}\in\mathcal{M}\}$. 
We assume that $\Prec$ and $\mathcal{M}$ are rectangular, so that any $\mu\in\mathcal{M}$ is a product of independent measures $\mu_{s}$ over $\Prec_{s}$. 

Given any policy $\pi$, the \emph{distributionally robust value function} is the worst-case expectation over the ambiguity set, \ie $\forall s\in\St$, 
$v_{\mathcal{M}}^\pi(s) =\inf_{\mu\in\mathcal{M}} \mathbb{E}_{p\sim \mu}[v^\pi_p(s)],$
and a distributionally robust optimal policy $\pi_{\mathcal{M}}^*$ satisfies
$\pi_{\mathcal{M}}^*\in \argmax_{\pi\in\Pi} v_{\mathcal{M}}^\pi$. Although the DRMDP setting can be very general, the ambiguity set must satisfy specific properties for the solution to be tractable \cite{xu2010distributionally, yu2015distributionally, chen2019distributionally}. In the following, we shall introduce DRMDPs with Wasserstein distance-based ambiguity sets that yield a solvable reformulation. This will enable us to connect robust MDPs with regularization in Sec.~\ref{section:reg_WDRMDP}, and to establish out-of-sample guarantees in Sec.~\ref{section:out_of_sample}.

\subsection{Wasserstein Distributionally Robust Markov Decision Process}
\label{section:wdrmdp}
The Wasserstein metric arises from the theory of optimal transport and measures the optimal transport cost between two probability measures \cite{villani2008optimal}. It can be viewed as the minimal cost required for turning a pile of sand into another, with the cost function being the amount of sand times its distance to destination. More formally, for any $s\in\St$, let $\norm{\cdot}$ be a norm on the uncertainty set $\Prec_s \subseteq \mathcal{M}(\St)^{\abs{\A}}$. We define the Wasserstein distance over distributions supported on $\Prec_s$ as follows. 

\begin{definition}
% [\textbf{Wasserstein metric}]
The \textbf{$1$-Wasserstein distance} between two probability measures $\mu_s,\nu_s\in\mathcal{M}(\Prec_s)$ is: 
$$
d(\mu_s, \nu_s) := \min_{\gamma\in\Gamma(\mu_s, \nu_s)}\left\{ \int_{\Prec_s\times \Prec_s}\norm{p_s - p_s'}\gamma(dp_s,dp_s')\right\},
$$
where $\Gamma(\mu_s, \nu_s)$ is the set of distributions over $\Prec_s\times \Prec_s$ with marginals $\mu_s$ and $\nu_s$.
\end{definition}

% \esther{Why 1-Wasserstein and not general p? In fact, same results should apply for $p<\infty$.}

Given $\hat{\mu}_s\in\mathcal{M}(\mathcal{P}_s)$, we can define the Wasserstein ball of radius $\alpha_s$ centered at $\hat{\mu}_s$ as:
$$
\mathfrak{M}_{\alpha_s}(\hat{\mu}_s):= \{\nu_s\in\mathcal{M}(\mathcal{P}_s): d(\hat{\mu}_s, \nu_s)\leq \alpha_s\},
$$
which leads us to introduce Wasserstein DRMDPs. In that setting, the construction of a contracting \emph{Wasserstein distributionally robust Bellman operator} enables to use standard planning algorithms for finding an optimal policy \cite{chen2019distributionally}.  

\begin{definition}
% [\textbf{Wasserstein DRMDP}]
A \textbf{Wasserstein DRMDP (WDRMDP)} is a tuple $\langle \St, \A, r, \Prec,\mathfrak{M}_{\alpha}(\hat{\mu}) \rangle$ with $\mathfrak{M}_{\alpha}(\hat{\mu}) = \bigotimes_{s\in\St} \mathfrak{M}_{\alpha_s}(\hat{\mu}_s)$ such that $p_s\in \Prec_s$ is unknown for all $s\in\St$. Instead, it is a random variable of distribution $\mu_s\in\mathfrak{M}_{\alpha_s}(\hat{\mu}_s)$. 
\end{definition}

\section{Regularization and Wasserstein DRMDPs} \label{section:reg_WDRMDP} 
This section addresses the core contributions of our study. We first define the \emph{empirical value function} as an MDP counterpart of the empirical risk defined in Sec. \ref{section:introduction}. Then, we introduce a regularization method whose connection with robustness is detailed in Sec.~\ref{subsection:robust_reg}.

\subsection{Empirical Value Function and Distribution}
\label{subsection:empirical_estimate}
% The nominal distribution $\hat{\mu}$ of a WDRMDP $\langle \St, \A, r, \Prec,\mathfrak{M}_{\alpha}(\hat{\mu}) \rangle$ is the center of a Wasserstein ball. 
In our setting, we take the empirical distribution as the center of the Wasserstein ball ambiguity set. We estimate it based on historical observations, as described below. 

\textbf{Tabular State-Space. }Given $n$ episodes of respective lengths $(T_i)_{i\in [n]}$ we estimate the transition model through visit counts as: $\widehat{p}_i(s'|s,a) := \frac{n_i(s,a,s')}{\sum_{s''\in\St}n_i(s,a,s'')}$,
where $s,s'\in\St, a\in\A, i\in[n]$ and $n_i(s,a,s')$ is the number of transitions $(s,a,s')$ occurred during episode $i$. 

\textbf{Large State-Space. }If the state space is too large to be stored in a table, we use kernel averages to approximate the empirical transition function \cite{lim2019kernel}. For all action $a\in\A$, define a kernel $\psi_a:\St\times\St \rightarrow \mathbb{R}_+$. Then, we define the empirical transition function for all $s,s'\in\St$ as:
$$
\widehat{p}_i(s'| s,a) := \frac{\psi_a(s,s')n_i(s,a,s')}{\sum_{s''\in\St}\psi_a(s,s'')n_i(s,a,s'')}, \quad\forall i\in[n].
$$

In both tabular and large state-spaces, a model estimate $\widehat{p}_i$ can be deduced from each episode, which yields an \emph{empirical value function} $v_{\hat{p}_i}^\pi$ for any policy $\pi\in\Pi$. Then, the empirical distribution over transition functions is defined as $\widehat{\mu}_n := \bigotimes_{(s,a)\in\St\times\A} \widehat{\mu}^n_{s,a}$ where for all $(s,a)\in\St\times\A, $
$
\widehat{\mu}^n_{s,a} := \frac{1}{n}\sum_{i=1}^n \delta_{\hat{p}_i(\cdot |s,a)}
$
and $\delta_{\hat{p}_i(\cdot |s,a)}$ is a Dirac distribution with full mass on $\hat{p}_i(\cdot |s,a)$.
Setting $\delta_i := \bigotimes_{(s,a)\in\St\times\A}\delta_{\hat{p}_i(\cdot| s,a)}$, the empirical distribution $\hat{\mu}_n$ can further be written as 
$\hat{\mu}_n = \frac{1}{n}\sum_{i=1}^n\delta_i$, which defines the center of our Wasserstein ball ambiguity set and enables to construct the WDRMDP $\langle \St, \A, r, \mathfrak{M}_\alpha(\hat{\mu}_n)\rangle$.

\subsection{Robustification via Regularization} \label{subsection:robust_reg}
We now hold the tools for establishing our main result, which we prove for tabular and large state-spaces.
% : The regularized empirical value function is a lower bound of a WDRMDP.

% \subsubsection{The Tabular Case}
\begin{theorem}[\textbf{Tabular Case}]
\label{theorem:tabular_equivalence}
Let the WDRMDP $\langle \St, \A, r, \mathfrak{M}_\alpha(\hat{\mu}_n)\rangle$ as defined above and $L_{\gamma,\beta, R_{\max}} := \frac{\beta\gamma R_{\max}}{(1-\gamma)^2}$ where $\beta\geq 0$. Then, for all $\pi\in\Pi, s\in\St,$ we have:
$v_{\mathfrak{M}_\alpha(\hat{\mu}_n)}^\pi(s) \geq \frac{1}{n}\sum_{i=1}^n\limits v^\pi_{\hat{p}_i}(s) - \kappa_s^{\pi} \alpha,$
where $\kappa_s^{\pi}\geq 0$ depends on $\pi$ and $s$ while satisfying $\kappa_s^{\pi} \leq L_{\gamma,\beta, R_{\max}}$.
\end{theorem}

From the above theorem we obtain that\\
% if $\mathfrak{Z}_s^{\pi}\subseteq \tilde{\mathfrak{Z}}^{\pi}$, then 
% $$
% v_{\mathfrak{M}_\alpha(\hat{\mu}_n)}^\pi(s) \geq \frac{1}{n}\sum_{i=1}^n v^\pi_{\hat{p}_i}(s) - L_{\gamma,\beta, R_{\max}} \alpha
% $$ 
% and since $v_{\mathfrak{M}_\alpha(\hat{\mu}_n)}^\pi(s) = \inf_{\mu \in\mathfrak{M}_\alpha(\hat{\mu}_n)} \mathbb{E}_{p\sim \mu}[v^\pi_p(s)]$ we obtain the following bound:
~$
\frac{1}{n}\sum_{i=1}^n\limits v^\pi_{\hat{p}_i}(s) \geq v_{\mathfrak{M}_\alpha(\hat{\mu}_n)}^\pi(s) \geq \frac{1}{n}\sum_{i=1}^n\limits v^\pi_{\hat{p}_i}(s) - L_{\gamma,\beta, R_{\max}} \alpha.$\\
% This inequality suggests to use the product $L_{\gamma,\beta, R_{\max}}\alpha$ as a practical regularizer that yields a lower bound of the distributionally robust value function. 
In particular, the regularized value function $\frac{1}{n}\sum_{i=1}^n v^\pi_{\hat{p}_i}(s) - L_{\gamma,\beta, R_{\max}} \alpha$ is guaranteed to be distributionally robust w.r.t. the Wasserstein ball $\mathfrak{M}_\alpha(\hat{\mu}_n)$ centered at the empirical distribution. This is of particular use for RL methods, as it enables to ensure better generalization without resorting the additional computations that DRMDPs require. As a matter of fact, standard value iteration can be performed to learn $v^\pi_{\hat{p}_i}(s)$ for all $i\in[n]$ and subtracting $L_{\gamma,\beta, R_{\max}}\alpha$ to the resulting average ensures distributional robustness w.r.t. the ambiguity set $\mathfrak{M}_\alpha(\hat{\mu}_n)$.

When the state-space is large, one generally approximates the value function using feature vectors. Specifically, define as $\Phi(\cdot)\in\mathbb{R}^m$ a feature vector function such that for all $p\in \Prec$ we have 
$
v^\pi_p(s) \approx  \Phi(s)^\top \w_p
$
and assume all feature vectors are linearly independent.
Under standard conditions, Thm. \ref{theorem:tabular_equivalence} generalizes to large scale MDPs using this linear value function approximation. 

\begin{theorem}[\textbf{Linear Approximation Case}]
\label{theorem:approximate_equivalence}
Let the WDRMDP $\langle \St, \A, r, \mathfrak{M}_\alpha(\hat{\mu}_n)\rangle$. Then, for all $\pi\in\Pi, s\in\St:$
$$\inf_{\mu\in\mathfrak{M}_\alpha(\widehat{\mu}_n)}\limits\mathbb{E}_{p\sim\mu}[\Phi(s)^\top\w_p] \geq \frac{1}{n}\sum_{i=1}^n \Phi(s)^\top \w_{\widehat{p}_i} - \eta_s^{\pi}\alpha,$$
where $\eta_s^{\pi}\geq 0$ depends on $s$ and $\pi$.
\end{theorem}

\section{Out-of-Sample Performance Guarantees}
\label{section:out_of_sample}
Assume that at each episode, a transition model $p$ has been generated by some unknown distribution $\mu$. A potential defect of non-robust MDP formulations is that optimal policies may perform poorly once deployed on new data. This section gives out-of-sample performance guarantees in WDRMDPs with carefully determined Wasserstein-ball radii. 

From a statistical learning viewpoint, transition model estimates can be seen as a training set $\widehat{\Prec}_n := (\widehat{p}_i)_{1\leq i\leq n}$ following a distribution $\mu^n$ supported on $\Prec^n$. To avoid cluttered notation, we shall denote by $\hat{\pi}^*:= \pi^*_{\mathfrak{M}_\alpha(\hat{\mu}_n)}$ an optimal policy for the WDRMDP induced by the training set and $\hat{v}^* := v^{\hat{\pi}^*}_{\mathfrak{M}_\alpha(\hat{\mu}_n)}$ the optimal distributionally robust value function. Then, the out-of-sample performance of $\hat{\pi}^*$ is given by $\E_{p\sim\mu}[v^{\hat{\pi}^*}_p(s)]$. 
Defining the event $A:= \left\{\widehat{p} \mid \E_{p\sim\mu}[v^{\hat{\pi}^*}_p(s)] \geq v^{\hat{\pi}^*}_{\mathfrak{M}_\alpha(\hat{\mu}_n)}(s), \quad \forall s\in\St\right\}$ with $v^{\hat{\pi}^*}_{\mathfrak{M}_{\alpha(n,\epsilon)}(\widehat{\mu}_n)}(s)$ representing the certificate of the out-of-sample performance, Thm.~\ref{theorem:out_of_sample} establishes that $\hat{\pi}^*$ satisfies
$
\mu^n\left(A\right) \geq 1-\epsilon,
$
where $\epsilon\in(0,1)$ is a confidence level
% , $A$ denotes the event
% $$
% A := \left\{\widehat{p} \mid \E_{p\sim\mu}[v^{\hat{\pi}^*}_p(s)] \geq v^{\hat{\pi}^*}_{\mathfrak{M}_\alpha(\hat{\mu}_n)}(s) \quad , \forall s\in\St\right\}
% $$
\cite{fournier2015rate}. This bound is the best we can hope for, as the true generating distribution $\mu$ is unknown.

\begin{theorem}[\textbf{Finite-sample Guarantee}]
\label{theorem:out_of_sample}
Let $\epsilon\in (0,1), m:= \abs{\St}\times\abs{\A}$ and $\langle \St, \A, r, \mathfrak{M}_{\alpha(n,\epsilon)}(\widehat{\mu}_n)\rangle$ a WDRMDP. Denote by $\hat{v}^*$ and $\hat{\pi}^*$ its optimal value and policy, respectively. If the radius of the Wasserstein ball at $s\in\St$ satisfies 
\begin{equation*}
\alpha_s(n_s,\epsilon) := \left\{
\begin{split}
    &c_0 \left(\frac{1}{n_s c_2}\log\left(\frac{c_1}{\epsilon}\right)\right)^{1/(m \vee 2)} &&\text{ if } n_s \geq C_{m}^\epsilon\\
    &c_0  && \text{ otherwise}
\end{split}
\right.
\end{equation*}
with $C_{m}^\epsilon= \frac{1}{c_2}\log\left(\frac{c_1}{\epsilon}\right)$ and $n_s = \sum_{i\in[n], a\in\A,s'\in\St}n_i(s,a,s')$, then 
$\mu^n\left(A\right)\geq 1-\epsilon$, where $c_0, c_1, c_2\in\mathbb{R}^+$ only depend on $m\neq 2$\footnote{A comparable conclusion can be established for $m=2$, but we omit it due to its limited interest.}. 
\end{theorem}

In the above theorem, $c_0$ corresponds to the diameter of the whole space $\mathcal{M}(\Prec_s)$. Therefore, if the sample size is smaller than $C_m^{\epsilon}$, then the WDRMDP as defined in Thm.~\ref{theorem:out_of_sample} becomes a robust MDP of uncertainty set $\Prec_s$. Additionally, since the radius $\alpha_s(n_s,\epsilon)$ tends to $0$ as $n_s$ goes to infinity, the solution becomes less conservative as the sample size increases. 
Thm. \ref{theorem:out_of_sample} extends \citet{yang2018wasserstein}[Thm. 3] to MDPs and ensures that with high probability, the distributionally robust optimal policy cannot yield lower value than the certificate performance $v^{\hat{\pi}^*}_{\mathfrak{M}_\alpha(\hat{\mu}_n)}(s)$ deduced from the training set. As a result, the regularized value function a fortiori satisfies this performance guarantee.

\section{Discussion}

Our study facilitates robust RL by establishing regularized value function as a lower-bound of a distributionally robust value. Future work should analyze the tightness and the asymptotic consistency of our approach with increasing sample size. Other compelling directions include the extension of our results to non-linear function approximation and deep architectures \cite{levine2017shallow}. 
It would also be interesting to consider extensions of our regularized formulation to policy optimization, and build a connection with regularized policy search. 
\section*{Acknowledgements}
The authors would like to thank Shimrit Shtern for pointers to the relevant literature. Thanks also to Shirli Di Castro Shashua, Guy Tennenholtz and Nadav Merlis for their comprehensive review of an earlier draft.

\bibliography{reg}
\bibliographystyle{icml2020}

\newpage
\onecolumn
\appendix

\section{Preliminaries on Convex Analysis}
For completeness, this section provides the theoretical background used throughout the proofs. We first recall some definitions and fundamental results of convex analysis. This preliminary study will play a crucial role in our regularization method. 

Consider a convex Euclidean space $X$ equipped with a scalar product $\innorm{\cdot, \cdot}$. Given a norm $\norm{\cdot}$ over $X$, the dual norm is defined through $\norm{\cdot}_* = \sup_{\norm{x}\leq 1}\innorm{\cdot, x}$. Further denote by $\overline{\mathbb{R}} = [-\infty, +\infty]$ the extended reals and $U: X \rightarrow \overline{\mathbb{R}}$ an extended real-valued function over $X$. We then define the following.

\begin{definition}
\begin{itemize*}
    \item [(a)] \textbf{Proper Function. } We say that $U$ is proper if $U >-\infty$ and there exists $x\in X $ such that $U(x) <+\infty$.\\
    \item [(b)] \textbf{Closed Function. } We say that $U$ is a closed function if its epigraph
    $\mathbf{epi}(U) := \{(x,c)\in X\times \mathbb{R} | U(x) \leq c\}
    $
is a closed subset of $X\times \mathbb{R}$.\\
\textbf{Convex Closure. } The convex closure $\Breve{\text{cl}}(U)$ of $U$ is the greatest closed and convex function upper-bounded by $U$ \ie if $U_c$ is a closed and convex function that satisfies $U_c\leq U$, then $U_c\leq \Breve{\text{cl}}(U)$.
\end{itemize*}
\end{definition}

In other words, a function is proper if and only if its epigraph is nonempty and does not contain a vertical line. Moreover, when dealing with a non-convex function, we may work with its convex closure instead, in order to apply standard results from convex analysis. In particular, if the convex closure of a function is proper, then it coincides with its double conjugate, as we detail below. 

\begin{definition}[\textbf{Conjugate Function}]
The Legendre-Fenchel transform (or conjugate function) of $U$ is the mapping $U^*: X \rightarrow \overline{\mathbb{R}}$ defined by 
$$
U^*(y) := \sup_{x\in X}\langle  y, x\rangle - U(x),
$$
Denoting by $X^*:= \{y: X\rightarrow \mathbb{R} | y \text{ is linear}\}$ the dual space of $X$, we further define the conjugate function of $U^{*}$ as
$$
U^{**}(x):=\sup_{y\in X^*}\langle y, x\rangle - U^*(y),
$$
which is also the double conjugate of $U$. 
\end{definition}

Regardless of the initial function $U$, its conjugate $U^*$ is convex and closed but not necessarily proper. In fact, if $U$ is convex, then $U^*$ is proper if and only if $U$ is, as stated in the fundamental theorem below (see \citet{bertsekas2009convex, barbu2012convexity} for a proof). 

\begin{theorem}[\textbf{Conjugacy Theorem}]
\label{theorem:conjugacy_thm}
The following holds:\\
\begin{itemize*}
    \item[(a)] $U \geq U^{**}$\\
    % \item[(b)] $U_1 \geq U_2$ implies $U_2^*\geq U_1^*$\\
    
    \item[(b)] If $U$ is convex and closed, then $U$ is proper if and only if $U^*$ is.\\
    
    \item[(c)] If $U$ is closed, proper and convex, then $U = U^{**}$.\\

    \item[(d)] The conjugates of $U$ and $\Breve{\text{cl}}(U)$ are equal. 
\end{itemize*}
\end{theorem}

Additionally to this standard theorem, we will be using the following result.

\begin{proposition}
\label{prop:convex_closure_proper}
If $X$ is compact and $U$ is a proper closed function over $X$, then its convex closure $\Breve{\text{cl}}(U)$ is also proper. 
\end{proposition}

\begin{proof}
We apply the Weierstrass theorem \cite{barbu2012convexity}[Thm. 2.8.]: Since $U$ is closed on $X$ compact, it takes a minimal value on $X$. Therefore, there exists $x_0\in X$ such that $\inf_{x\in X}U(x)= U(x_0)$. Moreover, we have $\inf_{x\in X}\Breve{\text{cl}}(U)(x) = \inf_{x\in X}U(x)$, so $\inf_{x\in X}\Breve{\text{cl}}(U)(x) = U(x_0)$. It follows that $\Breve{\text{cl}}(U) > -\infty$. On the other hand, since $U$ is proper, there exists $x\in X$ such that $U(x)<+\infty$. By definition of the convex closure, $\Breve{\text{cl}}(U)(x) \leq U(x)$ so $\Breve{\text{cl}}(U)$ is proper. 
\end{proof}

\section{Regularization and Wasserstein DRMDPs}
\subsection{Proof of Theorem \ref{theorem:tabular_equivalence}}
We first establish the following lemma. 

\begin{lemma}
\label{lemma_appendix:conjugate_proper}
For all policy $\pi\in\Pi$ and state $s\in\St$ define the mapping $u_s^\pi: \Prec \rightarrow \mathbb{R}$ as $p \mapsto v_{p}^\pi(s)$. Then, the following holds:\\
\begin{itemize*}
\item [(i)] $u_s^\pi$ is proper.\\
\item [(ii)] If $\Prec$ is closed, then $u_s^\pi$ is continuous and thus, it is a closed function.
% \item [(iii)] For $\Prec = \mathcal{M}(\St)^{\abs{\St} \times\abs{\A}}$, the convex closure $\Breve{\text{cl}}(u_s^\pi)$ of $u_s^\pi$ is proper.
\end{itemize*}
\end{lemma}

\begin{Proof}
\emph{\textbf{Claim (i).} }By assumption, the reward function is bounded by $R_{\max}$, so we have
$$
\abs*{v_{p}^\pi(s)} =  \abs*{\mathbb{E}_p^\pi\left[\sum_{t=0}^\infty \gamma^t r(s_t,a_t)| s_0=s\right]} \leq \sum_{t=0}^\infty \gamma^t R_{\max }= \frac{R_{\max}}{1-\gamma},
$$
and $u_s^\pi$ is proper.

\emph{\textbf{Claim (ii).} } Denote by $(p_n)_{n\geq 1}$ a sequence that converges to $p$. Since $\Prec$ is closed, $p\in\Prec$ and $u_s^\pi(p)$ is well defined. Moreover, $u_s^\pi(p)= v_p^\pi(s)$. For all $n\geq 1$ we introduce the Bellman operator  $T_{p_n}^\pi$ w.r.t. transition $p_n$:
$$
T_{p_n}^\pi v(s) =  r(s,  \pi(s)) + \gamma  \sum_{s'\in\St}p_n(s, \pi(s),s')v(s'),
$$
and $T_{p}^\pi$ the Bellman operator w.r.t. transition $p$:
$$
T_{p}^\pi v(s) =  r(s,  \pi(s)) + \gamma  \sum_{s'\in\St}p(s, \pi(s),s')v(s').
$$
Then, we have
\begin{equation*}
    \begin{split}
    \lim_{n\rightarrow\infty}T_{p_n}^\pi v(s)  &=  r(s,  \pi(s)) + \gamma \lim_{n\rightarrow\infty}   \sum_{s'\in\St}p_n(s, \pi(s),s')v(s')\\
    &\overset{(a)}{=}  r(s,  \pi(s)) + \gamma    \sum_{s'\in\St}\lim_{n\rightarrow\infty}p_n(s, \pi(s),s')v(s')\\
    &=  r(s,  \pi(s)) + \gamma    \sum_{s'\in\St}p(s, \pi(s),s')v(s')\\
    &= T^\pi_p v(s),
    \end{split}
\end{equation*}
where equality $(a)$ holds since $\St$ and $\A$ are finite sets. Remark that here, we established the continuity of the Bellman operator with respect to the transition function.

As a result, for all $\epsilon > 0$, there exists $n_{\epsilon}$ such that for all $n\geq n_{\epsilon}$ we have $\abs{T_{p_n}^\pi v_{p_n}^\pi(s) -T^\pi_p v_{p_n}^\pi(s)} \leq (1-\gamma)\epsilon$. Using the fact that $v_{p_n}^\pi$ (resp. $v_{p}^\pi$) is the unique fixed point of $T_{p_n}^\pi $ (resp. $T_{p}^\pi $) and that  $T_{p}^\pi$ is a $\gamma$-contraction, for all $n\geq n_{\epsilon}$ we can write:
\begin{equation*}
\begin{split}
  \abs{v_{p_n}^\pi(s) - v_{p}^\pi(s)} &= \abs{T_{p_n}^\pi v_{p_n}^\pi(s) - T_p^\pi v_{p}^\pi(s)} \\
  &\leq \abs{T_{p_n}^\pi v_{p_n}^\pi(s) -T^\pi_p v_{p_n}^\pi(s)}  +\abs{T^\pi_p v_{p_n}^\pi(s) -  T^\pi_p v_{p}^\pi(s)} \\
  &\leq (1-\gamma)\epsilon + \gamma \abs{ v_{p_n}^\pi(s) -   v_{p}^\pi(s)},
\end{split}    
\end{equation*}
so $(1-\gamma)\abs{v_{p_n}^\pi(s) - v_{p}^\pi(s)} \leq  (1-\gamma)\epsilon$. Since $\gamma\in (0,1)$, $(1-\gamma)$ is positive and dividing both sides by $(1-\gamma)$ yields $\abs{v_{p_n}^\pi(s) - v_{p}^\pi(s)} \leq \epsilon$. Based on the fact that $\epsilon >0$ is arbitrary, we have shown that $v_{p_n}^\pi(s) \rightarrow_{n\rightarrow\infty} v_{p}^\pi(s)$, which concludes the proof.

% \emph{Claim (iii). }
% By definition, $\Breve{\text{cl}}(u_s^\pi) \leq u_s^\pi$ and $u_s^\pi$ is proper by Claim $(i)$. Therefore, there exists $p\in \mathcal{M}(\St)^{\abs{\St} \times \abs{\A}}$ such that  $\Breve{\text{cl}}(u_s^\pi)(p) < +\infty$.
% Moreover, based on \cite{bertsekas2009convex}[Prop. 1.3.13.], we have
% $$
% \inf_{p\in \mathcal{M}(\St)^{\abs{\St} \times \abs{\A}}} u_s^\pi(p) = 
% \inf_{p\in \mathcal{M}(\St)^{\abs{\St} \times \abs{\A}}}\Breve{\text{cl}}(u_s^\pi)(p).
% $$
% Since $\mathcal{M}(\St)^{\abs{\St} \times \abs{\A}}$ is closed and compact, Claim $(ii)$ ensures that $u_s^\pi$ is continuous and thus, the infimum is a minimum besides being finite \ie 
% $$
% \min_{p\in \mathcal{M}(\St)^{\abs{\St} \times \abs{\A}}}u_s^\pi(p) = 
% \inf_{p\in \mathcal{M}(\St)^{\abs{\St} \times \abs{\A}}}\Breve{\text{cl}}(u_s^\pi)(p)> -\infty.
% $$
% In particular, $\Breve{\text{cl}}(u_s^\pi)(p) > -\infty$ for all $p\in\mathcal{M}(\St)^{\abs{\St} \times \abs{\A}}$, which concludes the proof.
\end{Proof}

We are now ready to prove Thm.\ref{theorem:tabular_equivalence}, whose full statement is recalled below: 
\begin{theorem*}[\textbf{Tabular Case}]
Let the WDRMDP $\langle \St, \A, r, \mathfrak{M}_\alpha(\hat{\mu}_n)\rangle$ as defined above and $L_{\gamma,\beta, R_{\max}} := \frac{\beta\gamma R_{\max}}{(1-\gamma)^2}$ where $\beta\geq 0$. Then, for all $\pi\in\Pi, s\in\St,$ we have:
$v_{\mathfrak{M}_\alpha(\hat{\mu}_n)}^\pi(s) \geq \frac{1}{n}\sum_{i=1}^n\limits v^\pi_{\hat{p}_i}(s) - \kappa_s^{\pi} \alpha,$
where $\kappa_s^{\pi}\geq 0$ depends on $\pi$ and $s$ while satisfying $\kappa_s^{\pi} \leq L_{\gamma,\beta, R_{\max}}$.
\end{theorem*}

The proof proceeds in two steps. First, we establish the regularized value function as a lower bound of the distributionally robust value function. Then, we provide an upper bound of the regularization term, thus enabling to determine the maximal gap between both values.

\textbf{Claim (i). Let $\langle \St, \A, r, \mathfrak{M}_\alpha(\hat{\mu}_n)\rangle$ be a Wasserstein DRMDP with a radius-$\alpha$-ball ambiguity set centered at the empirical distribution (see its construction in Sec.~\ref{subsection:empirical_estimate}). Then, for all $\pi\in\Pi, s\in\St,$ we have:
$v_{\mathfrak{M}_\alpha(\hat{\mu}_n)}^\pi(s) \geq \frac{1}{n}\sum_{i=1}^n\limits v^\pi_{\hat{p}_i}(s) - \kappa_s^{\pi} \alpha,$
where $\kappa_s^{\pi}\geq 0$ depends on $\pi$ and $s$.}

\begin{proof}
% Without loss of generality, we consider the stationary model formulation where the distribution over transitions is initially chosen by Nature and remains fixed thereafter. Indeed, dynamic and stationary models are equivalent when the horizon is infinite, as depicted in \cite{nilim2005robust, xu2010distributionally,chen2019distributionally}. The stationary model is given by  
% $$
% \mathcal{M}^\textsc{s}(\alpha) = \left\{\tilde{\mu} \biggr | \tilde{\mu} = \bigotimes_{t\geq 0}\mu_{t};  \mu_{t} = \mu ,\quad \forall t\geq 0; \mu \in  \mathfrak{M}_\alpha(\hat{\mu}_n)\right\}, 
% $$
% and we have 
For all $i\in [n]$, let $\hat{p}^i_s := \bigotimes_{a\in\A}\hat{p}^i_{s,a}$ and $\hat{\mu}_{s}^n = \frac{1}{n}\sum_{i=1}^n \delta_{\hat{p}^i_s}$, so that $\hat{\mu}_n := \bigotimes_{s\in\St}\hat{\mu}_s^n$. 
The Wasserstein distributionally robust value function is given by
$$
v^\pi_{\mathfrak{M}_\alpha(\hat{\mu}_n)}(s)  = \inf_{\mu \in  \mathfrak{M}_\alpha({\hat{\mu}}_n)} \mathbb{E}_{p\sim \mu}[v^\pi_p(s)].
$$
By construction of the ambiguity set $\mathfrak{M}_\alpha(\hat{\mu}_n)$ and the empirical distribution $\hat{\mu}_n$, the constraint $\mu \in  \mathfrak{M}_\alpha(\hat{\mu}_n)$ is equivalent to requiring $\mu_s \in  \mathfrak{M}_{\alpha_s}(\hat{\mu}^n_s)$,\ie $d(\mu_s, \hat{\mu}_{s}^n) \leq \alpha_s$ for all $s\in\St$. 
% Moreover, the constraint $\tilde{\mu}\in\mathcal{M}^\textsc{s}(\alpha)$ corresponds to the following: 
% \begin{equation*}
% \left\{
% \begin{split}
% &\tilde{\mu} = \bigotimes_{t\geq 0}\mu_{t};\\ 
% &\mu_{t} = \mu ,\quad \forall t\geq 0;  \\
% &\mu = \bigotimes_{s\in\St}\mu_{s}; \\
% &\mu_s \in  \mathfrak{M}_{\alpha_s}(\hat{\mu}^n_s) ,\quad \forall s\in\St.
% \end{split}
% \right.
% \end{equation*}
% Denote $\hat{p}^i_s := \bigotimes_{a\in\A}\hat{p}^i_{s,a}$. By definition of the ambiguity set, $\mu_s\in \mathfrak{M}_{\alpha_s}(\hat{\mu}_{s}^n)$ if and only if $d(\mu_s, \hat{\mu}_{s}^n) \leq \alpha_s$. 
Therefore, recalling the definition of the Wasserstein metric 
$$
d(\mu_s, \hat{\mu}_{s}^n) := \min_{\gamma \in \Gamma(\mu_s, \hat{\mu}_{s}^n)} \left\{  \int_{\Prec_s\times\Prec_s} \norm{p_s - p'_{s}}\gamma(dp_s, dp'_{s}) \right\},
$$
we have $\mu_s \in  \mathfrak{M}_{\alpha_s}(\hat{\mu}^n_s)$ if and only if 
% and the empirical distribution $\hat{\mu}_{s}^n = \frac{1}{n}\sum_{i=1}^n \delta_{\hat{p}_i(\cdot|s,a)}$, the constraint
% $\mu_s\in \mathfrak{M}_{\alpha_s}(\hat{\mu}_{s}^n)$ is equivalent to the
there exists $\mu^1_s,\cdots,\mu^n_s\in\mathcal{M}(\Prec_s)$ such that $\mu_s = \frac{1}{n}\sum_{i=1}^n \mu^i_s$ and 
$
\frac{1}{n}\sum_{i=1}^n \mathbb{E}_{p_s^i\sim\mu_s^i} \left[\norm{p_s^i-\widehat{p}_s^i}\right] \leq \alpha_s.
$

For all $i\in[n]$, define $p_i := \bigotimes_{s\in\St}p_s^i$, $\mu_i := \bigotimes_{s\in\St}\mu_s^i$, and the average distribution $\bar{\mu} := \frac{1}{n}\sum_{i=1}^n \mu_i$. Further consider the product space $\mathcal{M}(\St)^{\abs{\St}\times\abs{\A}}$ with the product norm corresponding to $\norm{\cdot}$ \eg if $\norm{\cdot} = \norm{\cdot}_2$ on $\Prec_s$, take the $\ell_2$-norm on $\Prec$ defined as $\norm{p}_2 := \left(\sum_{s\in\St}\norm{p_s}_2^2\right)^{1/2}$. Then, with the slight abuse of notation $\norm{p} \equiv \norm{p_s}$, the worst-case distributionally robust value function may be formulated as 
\begin{equation*}
    \begin{split}
        \inf_{\mu \in  \mathfrak{M}_\alpha({\hat{\mu}}_n)} \mathbb{E}_{p\sim \mu}[v^\pi_p(s)]  
        = \min_{\mu} \mathbb{E}_{p\sim \mu}[v^\pi_p(s)] \text{ s.t. } 
        \left\{
        \begin{split}
        &\mu = \bar{\mu}\\
        &\frac{1}{n}\sum_{i=1}^n \mathbb{E}_{p_i\sim\mu_i} \left[\norm{p_i-\widehat{p}_i}\right] \leq \alpha,
        \end{split}
        \right.
    \end{split}
\end{equation*}
for a radius $\alpha$ determined by radii $\alpha_s$-s. Thus, replacing distribution $\mu$ by its constrained law and using a duality argument, we obtain 
\begin{equation*}
    \begin{split}
        \inf_{\mu \in  \mathfrak{M}_\alpha({\hat{\mu}}_n)} \mathbb{E}_{p\sim \mu}[v^\pi_p(s)] 
        = \inf_{{\mu}_s^1,\cdots,{\mu}_s^n:
        \mu_i = \bigotimes_{s\in\St}\mu_s^i }\sup_{\lambda\geq 0} \left(\mathbb{E}_{p\sim \bar{\mu}}[v^\pi_p(s)]- \lambda \left(\alpha - \frac{1}{n}\sum_{i=1}^n\mathbb{E}_{p_i\sim\mu_i} \left[\norm{p_i-\widehat{p}_i}\right]\right)\right).
    \end{split}
\end{equation*}

Thanks to the maxmin inequality and setting all  $\mu^i_s$-s to a Dirac distribution with full mass on the worst-case transition $p\in \mathcal{M}(\St)^{\abs{\St}\times\abs{\A}}$, we can write 
\begin{equation*}
    \begin{split}
        \inf_{\mu \in  \mathfrak{M}_\alpha({\hat{\mu}}_n)} \mathbb{E}_{p\sim \mu}[v^\pi_p(s)]   
        &\geq \sup_{\lambda\geq 0} \inf_{{\mu}_s^1,\cdots,{\mu}_s^n:\mu_i = \bigotimes_{s\in\St}\mu_s^i} \left(\mathbb{E}_{p\sim \bar{\mu}}[v^\pi_p(s)]- \lambda \left(\alpha - \frac{1}{n}\sum_{i=1}^n\mathbb{E}_{p_i\sim\mu_i} \left[\norm{p_i-\widehat{p}_i}\right]\right)\right)\\
        &= \sup_{\lambda\geq 0} \frac{1}{n}\sum_{i=1}^n  \inf_{\mu_i: \mu_i = \bigotimes_{s\in\St}\mu_s^i}\left(\mathbb{E}_{p_i\sim \mu_i}\left[v^\pi_{p_i}(s)+ \lambda \norm{p_i-\widehat{p}_i}\right]\right)-\lambda \alpha\\
        &=  \sup_{\lambda\geq 0} \frac{1}{n}\sum_{i=1}^n  \inf_{p \in  \mathcal{M}(\St)^{\abs{\St}\times\abs{\A}}}\left(v^\pi_{p}(s)+ \lambda\norm{p-\widehat{p}_i}\right)-\lambda \alpha.
    \end{split}
\end{equation*}

Now, given $u_s^\pi: p\mapsto v_p^\pi(s)$, we define $\tilde{u}^{\pi}_s \colon \mathbb{R}^{\abs{\St}\times\abs{\A}\times\abs{\St}}\rightarrow\mathbb{R}$ as
\begin{alignat*}{3}
        \tilde{u}^{\pi}_s (\tilde{p}):=\inf_{p\in\mathcal{M}(\St)^{\abs{\St}\times\abs{\A}}} u^{\pi}_s(p) + L_{\gamma,\beta, R_{\max}}\norm{p-\tilde{p}}.
\end{alignat*}
This function will be used in the sequel to bound the gap between the distributionally robust value function and the regularized one. Also, it is clear that $\tilde{u}^{\pi}_s$ is a continuation function of $u_s^\pi$, 
so can write
\begin{equation*}
    \begin{split}
        \inf_{\mu \in  \mathfrak{M}_\alpha({\hat{\mu}}_n)} \mathbb{E}_{p\sim \mu}[v^\pi_p(s)]  
        &\geq  \sup_{\lambda\geq 0} \frac{1}{n}\sum_{i=1}^n  \inf_{p \in  \mathcal{M}(\St)^{\abs{\St}\times\abs{\A}}}\left(v^\pi_{p}(s)+ \lambda\norm{p-\widehat{p}_i}\right)-\lambda \alpha\\
        &\geq \sup_{\lambda\geq 0} \frac{1}{n}\sum_{i=1}^n  \inf_{\tilde{p} \in  \mathbb{R}^{\abs{\St}\times\abs{\A}\times\abs{\St}}}\left(\tilde{u}^{\pi}_s (\tilde{p})+ \lambda\norm{\tilde{p}-\widehat{p}_i}\right)-\lambda \alpha.
    \end{split}
\end{equation*}

We introduce auxiliary variables $x_1,\cdots,x_n$ in order to reformulate the bound as  
\begin{equation*}
    \begin{split}
        &\inf_{\mu \in  \mathfrak{M}_\alpha({\hat{\mu}}_n)} \mathbb{E}_{p\sim \mu}[v^\pi_p(s)] 
        \geq \max_{\lambda, x_1,\cdots,x_n} \frac{1}{n}\sum_{i=1}^n x_i-\lambda \alpha \text{ s.t. } \left\{
        \begin{split}
            &\inf_{\tilde{p} \in  \mathbb{R}^{\abs{\St}\times\abs{\A}\times\abs{\St}}}\left(\tilde{u}^{\pi}_s (\tilde{p})+ \lambda
            \norm{\tilde{p}-\widehat{p}_i}\right) \geq x_i ,\quad \forall i\in[n] \\
            &\lambda \geq 0.
        \end{split}\right.
    \end{split}
\end{equation*}
For all $i\in[n]$, let the mapping $f_i: \tilde{p} \mapsto \tilde{u}_s^\pi(\tilde{p})+ \lambda\norm{\tilde{p}-\widehat{p}_i}$. By applying Prop. 1.3.17. of \cite{bertsekas2009convex}, we obtain $\Breve{\text{cl}}(f_i)(\tilde{p})= \Breve{\text{cl}}(\tilde{u}_s^\pi)(\tilde{p})+ \lambda
            \norm{\tilde{p}-\widehat{p}_i}$, where we used the fact that $\tilde{p} \mapsto \lambda\norm{\tilde{p}-\widehat{p}_i}$ is convex and closed for all $\lambda \geq 0$. 
% Let the convex closure $\Breve{\text{cl}}(\tilde{u}_s^\pi)$ of $\tilde{u}_s^{\pi}$. For all $i\in[n]$, the function $\tilde{p} \mapsto \norm{\tilde{p}-\widehat{p}_i}$ is convex and closed. Thus, for all $\lambda \geq 0$, Let the mapping $f: \tilde{p} \mapsto \tilde{u}_s^\pi(\tilde{p})+ \lambda\norm{\tilde{p}-\widehat{p}_i}$ is convex and $\Breve{\text{cl}}(f)(\tilde{p})= \Breve{\text{cl}}(\tilde{u}_s^\pi)(\tilde{p})+ \lambda
%             \norm{\tilde{p}-\widehat{p}_i} $ \cite{bertsekas2009convex}[Prop. 1.3.17.].
By Prop. 1.3.13. of \cite{bertsekas2009convex}, we thus have 
\begin{equation*}
    \begin{split}
        % &\inf_{\mu \in\mathcal{M}_\infty} \mathbb{E}_{p\sim \mu}[v^\pi_p(s)]  \\
        &\max_{\lambda, x_1,\cdots,x_n} \frac{1}{n}\sum_{i=1}^n x_i-\lambda \alpha \text{ s.t. }
        \left\{
        \begin{split}
        &\inf_{\tilde{p} \in \mathbb{R}^{\abs{\St}\times\abs{\A}\times\abs{\St}}} \left(\tilde{u}^{\pi}_s(\tilde{p})+ \lambda
            \norm{\tilde{p}-\widehat{p}_i}\right) \geq x_i ,\quad \forall i\in[n]\\
        &\lambda \geq 0
        \end{split}
        \right.\\
    &= 
    \begin{split}
        \max_{\lambda, x_1,\cdots,x_n} \frac{1}{n}\sum_{i=1}^n x_i-\lambda \alpha \text{ s.t. }
        \left\{
        \begin{split}
        &\inf_{\tilde{p} \in \mathbb{R}^{\abs{\St}\times\abs{\A}\times\abs{\St}}} \left(\Breve{\text{cl}}(\tilde{u}_s^\pi)(\tilde{p})+ \lambda
            \norm{\tilde{p}-\widehat{p}_i}\right) \geq x_i ,\quad \forall i\in[n]\\
        &\lambda \geq 0.
        \end{split}
        \right.
    \end{split}
    \end{split}
\end{equation*}
Moreover, using the definition of the dual norm $\norm{\cdot}_*$, the inequality constraints are equivalent to the following:
\begin{equation*}
    \begin{split}
        \max_{\lambda, x_1,\cdots,x_n} \frac{1}{n}\sum_{i=1}^n x_i-\lambda \alpha \text{ s.t. }
        \left\{
        \begin{split}
        &\inf_{\tilde{p} \in \mathbb{R}^{\abs{\St}\times\abs{\A}\times\abs{\St}}}\sup_{\norm{y_i}_*\leq \lambda} \left(\Breve{\text{cl}}(\tilde{u}_s^\pi)(\tilde{p})+ \innorm{y_i, \tilde{p}-\widehat{p}_i} \right) \geq x_i ,\quad \forall i\in[n]\\
        &\lambda \geq 0,
        \end{split}
        \right.
    \end{split}
\end{equation*}
so the worst-case expectation can be reformulated as 
\begin{equation*}
    \begin{split}
        \inf_{\mu \in  \mathfrak{M}_\alpha({\hat{\mu}}_n)} \mathbb{E}_{p\sim \mu}[v^\pi_p(s)]
        \geq 
        \begin{split}
        \max_{\lambda, x_1,\cdots,x_n} \frac{1}{n}\sum_{i=1}^n x_i-\lambda \alpha \text{ s.t. }
        \left\{
        \begin{split}
        &\inf_{\tilde{p} \in \mathbb{R}^{\abs{\St}\times\abs{\A}\times\abs{\St}}}\sup_{\norm{y_i}_*\leq \lambda} \left(\Breve{\text{cl}}(\tilde{u}_s^\pi)(\tilde{p})+ \innorm{y_i, \tilde{p}-\widehat{p}_i} \right) \geq x_i ,\quad \forall i\in[n]\\
        &\lambda \geq 0.
        \end{split}
        \right.
    \end{split}
    \end{split}
\end{equation*}

Now introduce the conjugate transform of $\Breve{\text{cl}}(\tilde{u}_s^\pi)$ w.r.t. uncertainty set
$\Prec := \mathbb{R}^{\abs{\St}\times\abs{\A}\times\abs{\St}}$:
\begin{equation*}
 \begin{split}
\Breve{\text{cl}}(\tilde{u}_s^\pi)^*(z) := \sup_{\tilde{p}\in\mathbb{R}^{\abs{\St}\times\abs{\A}\times\abs{\St}}}\left(\langle z, \tilde{p}\rangle - \Breve{\text{cl}}(\tilde{u}_s^\pi)(\tilde{p})\right).
\end{split}   
\end{equation*}
By Thm.~\ref{theorem:conjugacy_thm}(d), $\Breve{\text{cl}}(\tilde{u}_s^\pi)^*(z) = (\tilde{u}_s^\pi)^*(z)$. Moreover, by Lemma \ref{lemma_appendix:conjugate_proper} and by construction of its continuation function, $\tilde{u}_s^\pi$ is proper. Therefore, its convex closure $\Breve{\text{cl}}(\tilde{u}_s^\pi)$ is also proper: Indeed, if it were not, then it would be identically equal to $-\infty$. Moreover, we have $\Breve{\text{cl}}(u_s^\pi)(p) \leq u_s^\pi(p) = (\tilde{u}_s^\pi)(p)$ for all $p\in\mathcal{M}(\St)^{\abs{\St}\times\abs{\A}}$, and by definition of $\Breve{\text{cl}}(u_s^\pi)$ as the greatest closed and convex minorant of $\tilde{u}_s^\pi$, we end up with $\Breve{\text{cl}}(u_s^\pi)(p) \leq \Breve{\text{cl}}(\tilde{u}_s^\pi)(p)$. As a result, we have $\Breve{\text{cl}}(u_s^\pi)(p) = -\infty$. This cannot happen thanks to Prop.~\ref{prop:convex_closure_proper}. 

Finally, according to Thm.~\ref{theorem:conjugacy_thm}(c), $\Breve{\text{cl}}(\tilde{u}_s^\pi)$ coincides with its bi-conjugate function and
\begin{equation*}
    \begin{split}
    \Breve{\text{cl}}(\tilde{u}_s^\pi)(\tilde{p}) &= \Breve{\text{cl}}(\tilde{u}_s^\pi)^{**}(\tilde{p})
    = \sup_{z_i\in \mathfrak{Z}_s^{\pi}}\left(\langle z, \tilde{p}\rangle - \Breve{\text{cl}}(\tilde{u}_s^\pi)^*(z)\right)
    = \sup_{z_i\in \mathfrak{Z}_s^{\pi}}\left(\langle z, \tilde{p}\rangle - (\tilde{u}_s^\pi)^*(z)\right),
    % &=\max_{z\in \mathfrak{Z}_s^{\pi}} \innorm{z, \tilde{p}}+  {v}^{*,\pi}_s(z) 
    \end{split}
\end{equation*}
where $\mathfrak{Z}_s^{\pi} := \{z:  \Breve{\text{cl}}(\tilde{u}_s^\pi)^*(z) < \infty\} =  \{z: (\tilde{u}_s^\pi)^*(z) < \infty\} $ is the effective domain of $(\tilde{u}_s^\pi)^*$.
Thus, if we use the reformulation of the convex closure and apply the minimax theorem \footnote{Prop. 5.5.4. of \cite{bertsekas2009convex}} 
we obtain
\begin{equation*}
    \begin{split}
    &\inf_{\tilde{p} \in \mathbb{R}^{\abs{\St}\times\abs{\A}\times\abs{\St}}}\sup_{\norm{y_i}_*\leq \lambda} \left(\Breve{\text{cl}}(\tilde{u}_s^\pi)(\tilde{p})+ \innorm{y_i, \tilde{p}-\widehat{p}_i} \right)  \\
    &= \inf_{\tilde{p} \in \mathbb{R}^{\abs{\St}\times\abs{\A}\times\abs{\St}}} \sup_{z_i\in \mathfrak{Z}_s^{\pi}}\sup_{\norm{y_i}_*\leq \lambda}  \innorm{z_i,\tilde{p}}-(\tilde{u}_s^\pi)^*(z_i) + \innorm{y_i, \tilde{p} - \widehat{p}_i}\\
    &=\sup_{z_i\in \mathfrak{Z}_s^{\pi}} \sup_{\norm{y_i}_*\leq \lambda} \inf_{\tilde{p} \in \mathbb{R}^{\abs{\St}\times\abs{\A}\times\abs{\St}}}   \innorm{z_i,\tilde{p}}-(\tilde{u}_s^\pi)^*(z_i) + \innorm{y_i,  \tilde{p} - \widehat{p}_i}\\
    &=\sup_{z_i\in \mathfrak{Z}_s^{\pi}} \sup_{\norm{y_i}_*\leq \lambda} -(\tilde{u}_s^\pi)^*(z_i)- \innorm{y_i, \widehat{p}_i} + \inf_{\tilde{p} \in \mathbb{R}^{\abs{\St}\times\abs{\A}\times\abs{\St}}}\innorm{\tilde{p}, z_i + y_i} \\
    &= \sup_{z_i\in \mathfrak{Z}_s^{\pi}} \sup_{\norm{y_i}_*\leq \lambda} -(\tilde{u}_s^\pi)^*(z_i) - \innorm{y_i, \widehat{p}_i}- \sigma_{\mathbb{R}^{\abs{\St}\times\abs{\A}\times\abs{\St}}}(-z_i - y_i) ,
    \end{split}
\end{equation*}
where $\sigma_{\Prec}(y) := \sup_{\tilde{p}\in \Prec}\innorm{\tilde{p},y}$ denotes the support function of a general set $\Prec$. 
% \emph{\textbf{Lower Bound. }}
We then
%use the conservative bound $\sigma_{\mathbb{R}^{\abs{\St}\times\abs{\A}\times\abs{\St}}} \leq \sigma_{\mathbb{R}^{\abs{\St}\times\abs{\A}\times\abs{\St}}}$ to
deduce that
\begin{equation*}
    \begin{split}
    &\inf_{\tilde{p} \in \mathbb{R}^{\abs{\St}\times\abs{\A}\times\abs{\St}}}\sup_{\norm{y_i}_*\leq \lambda} \left(\Breve{\text{cl}}(\tilde{u}_s^\pi)(\tilde{p})+ \innorm{y_i, \tilde{p}-\widehat{p}_i} \right) \\
    &= \sup_{z_i\in \mathfrak{Z}_s^{\pi}} \sup_{\norm{y_i}_*\leq \lambda} -(\tilde{u}_s^\pi)^*(z_i) - \innorm{y_i, \widehat{p}_i} - \sigma_{\mathbb{R}^{\abs{\St}\times\abs{\A}\times\abs{\St}}}(-z_i- y_i) \\
    % &\geq \sup_{z_i\in \mathfrak{Z}_s^{\pi}} \sup_{\norm{y_i}_*\leq \lambda} {v}^{*,\pi}_s(z_i) - \innorm{y_i, \widehat{p}_i} - \sigma_{{\mathbb{R}^{\abs{\St}\times\abs{\A}\times\abs{\St}}}}(-z_i - u_i) \\
    &= \sup_{z_i\in \mathfrak{Z}_s^{\pi}} \sup_{\norm{z_i}_*\leq \lambda}
    -(\tilde{u}_s^\pi)^*(z_i) + \innorm{z_i, \widehat{p}_i}\\
    &= \left\{
    \begin{split}
    \Breve{\text{cl}}(\tilde{u}_s^\pi)(\hat{p}_i) &\text{ if } \sup\{ \norm{z_i}_*: z_i \in\mathfrak{Z}_s^{\pi}\} \leq \lambda \\
    -\infty &\text{ otherwise}.
    \end{split} \right.
    \end{split}
\end{equation*}

Therefore, recalling the notation $\kappa_s^{\pi} := \sup\{\norm{z}_*: z\in\mathfrak{Z}_s^{\pi}\}$, we obtain
\begin{equation*}
    \begin{split}
        &\sup_{\lambda, x_1,\cdots,x_n} \frac{1}{n}\sum_{i=1}^n x_i-\lambda \alpha \text{ s.t. } \left\{
        \begin{split}
            &\inf_{\tilde{p} \in \mathbb{R}^{\abs{\St}\times\abs{\A}\times\abs{\St}}}\sup_{\norm{y_i}_*\leq \lambda} \left(\Breve{\text{cl}}(\tilde{u}_s^\pi)(\tilde{p})+ \innorm{y_i, \tilde{p}-\widehat{p}_i} \right) \geq x_i , \quad\forall i\in[n]\\
            &\lambda \geq 0
        \end{split}\right.\\
        &\geq \sup_{\lambda} \sup_{x_1,\cdots,x_n}\frac{1}{n}\sum_{i=1}^n x_i-\lambda \alpha \text{ s.t. } \left\{
        \begin{split}
            &\Breve{\text{cl}}(\tilde{u}_s^\pi)(\hat{p}_i)\geq x_i , \quad\forall i\in[n]\\
            &\lambda \geq \kappa_s^{\pi}.
        \end{split}\right.
    \end{split}
\end{equation*}
Putting this altogether yields
$ v^\pi_{\mathfrak{M}_{\alpha}(\hat{\mu}_n)}(s)  \geq  \frac{1}{n}\sum_{i=1}^n \Breve{\text{cl}}(\tilde{u}_s^\pi)(\hat{p}_i)  - \kappa_s^{\pi} \alpha$. 
Now let $F: (\hat{p}_1,\cdots, \hat{p}_n)\mapsto \sum_{i=1}^n \tilde{u}_s^\pi(\hat{p}_i)$. By Prop.1.3.17 of \cite{bertsekas2009convex}, we have $\Breve{\text{cl}}(F)(\hat{p}_1,\cdots, \hat{p}_n) = \sum_{i=1}^n \Breve{\text{cl}}(\tilde{u}_s^\pi)(\hat{p}_i)$ and since 
$f\leq C$ if and only if $\Breve{\text{cl}}(f) \leq C$ for any function $f$ and constant $C$, we obtain:
$$
v^\pi_{\mathfrak{M}_{\alpha}(\hat{\mu}_n)}(s) + \kappa_s^{\pi} \alpha \geq \frac{1}{n}\sum_{i=1}^n \tilde{u}_s^\pi(\hat{p}_i).
$$
Recalling that $\tilde{u}_s^\pi(\hat{p}_i) = v_{\hat{p}_i}^{\pi}(s)$ enables to establish 
$v^\pi_{\mathfrak{M}_{\alpha}(\hat{\mu}_n)}(s) \geq \frac{1}{n}\sum_{i=1}^nv_{\hat{p}_i}^{\pi}(s)-\kappa_s^{\pi} \alpha$.
% $\Breve{\text{cl}}(\tilde{u}_s^\pi)(\hat{p}_i) \leq \tilde{u}_s^\pi(\hat{p}_i)$ and $\tilde{u}_s^\pi(\hat{p}_i) = v^\pi_{\hat{p}_i}(s)$..
\end{proof}

\textbf{Claim (ii). For all $p\in\mathcal{M}(\St)^{\abs{\St}\times\abs{\A}}$, let the norm $\norm{p_s}_{\infty, 1}:= \max_{a\in\A} \sum_{s'\in\St}\abs{p(s,a,s') - p'(s,a,s')}$ and $L_{\gamma,\beta, R_{\max}} := \frac{\beta\gamma R_{\max}}{(1-\gamma)^2}$ where $\beta$ satisfies $\sum_{s\in\St}\norm{p_s}_{\infty, 1} \leq \beta \norm{p}$. Also recall $\tilde{u}^{\pi}_s(\tilde{p}) \colon \mathbb{R}^{\abs{\St}\times\abs{\A}\times\abs{\St}}\rightarrow\mathbb{R}$ the continuation function of $u_s^\pi$ defined as 
$\tilde{u}^{\pi}_s (\tilde{p}):=\inf_{p\in\mathcal{M}(\St)^{\abs{\St}\times\abs{\A}}} u^{\pi}_s(p) + L_{\gamma,\beta, R_{\max}}\norm{p-\tilde{p}}$, and
 $\mathfrak{Z}^{\pi} := \{z: (\tilde{u}_s^\pi)^*(z) < \infty\} $. Then, we have $\kappa_s^{\pi}\leq L_{\gamma,\beta, R_{\max}}$.}

\begin{proof}
We first introduce the following result of \citet{strehl2008analysis}[Lemma 1]:  
\begin{lemma}
\label{lemma:mbie_bound}
For all $p, p' \in \mathcal{M}(\St)^{\abs{\St}\times\abs{\A}}, \pi\in\Pi$ and $s\in \St$, we have 
$$
\abs{v_{p}^\pi(s) - v_{p'}^\pi(s)} \leq \frac{\gamma R_{\max} \norm{p_s- p'_s}_{\infty, 1}}{(1-\gamma)^2}.
$$
\end{lemma}

By Lemma \ref{lemma:mbie_bound}, for all $p, p' \in\mathcal{M}(\St)^{\abs{\St}\times\abs{\A}}$ we have 
$$
u_s^\pi(p') - L_{\gamma,\beta, R_{\max}} \norm{p - p'} \leq u_s^\pi(p) \leq u_s^\pi(p')+ L_{\gamma,\beta, R_{\max}} \norm{p - p'},
$$
so $u_s^\pi$ is $L_{\gamma,\beta, R_{\max}}$-Lipschitz continuous over $\mathcal{M}(\St)^{\abs{\St}\times\abs{\A}}$, and by construction, so is its continuation function $\tilde{u}_{s}^\pi$ over $\mathbb{R}^{\abs{\St}\times\abs{\A}\times\abs{\St}}$.
Therefore, if we fix a transition function $p\in \mathcal{M}(\St)^{\abs{\St}\times\abs{\A}}$ it holds that
\begin{equation*}
    \begin{split}
        (\tilde{u}_{s}^\pi)^*(z) &:= \sup_{\tilde{p}\in\mathbb{R}^{\abs{\St}\times\abs{\A}\times\abs{\St}}}\tilde{u}_{s}^\pi(\tilde{p})- \innorm{z,\tilde{p}}\\
        &\geq \sup_{\tilde{p}\in\mathbb{R}^{\abs{\St}\times\abs{\A}\times\abs{\St}}} v_{p}^\pi(s) -L_{\gamma,\beta, R_{\max}} \norm{p - \tilde{p}}- \innorm{z,\tilde{p}}\\
        &= \sup_{\tilde{p} \in \mathbb{R}^{\abs{\St}\times\abs{\A}\times\abs{\St}}} v_{p}^\pi(s)  -\sup_{\norm{x}_* \leq L_{\gamma,\beta, R_{\max}}} \langle x,\tilde{p}-p\rangle- \innorm{z,\tilde{p}}\\
        &= \sup_{\tilde{p} \in \mathbb{R}^{\abs{\St}\times\abs{\A}\times\abs{\St}}}\inf_{\norm{x}_* \leq L_{\gamma,\beta, R_{\max}}}v_{p}^\pi(s) + \langle x,  p-\tilde{p}  \rangle - \innorm{z,\tilde{p}}.
    \end{split}
\end{equation*}
Using the minimax theorem on the right hand side of the inequality, we obtain
\begin{equation*}
    \begin{split}
        (\tilde{u}_{s}^\pi)^*(z) &\geq 
        % \max_{\norm{x}_* 
        %  \leq C}\min_{p\in\Prec} v_{p_1}^\pi(s)+\langle x, p - p_1\rangle - \innorm{z,p}\\
         \inf_{\norm{x}_* \leq L_{\gamma,\beta, R_{\max}}}\sup_{\tilde{p}\in\mathbb{R}^{\abs{\St}\times\abs{\A}\times\abs{\St}}} v_{p}^\pi(s)+ \langle x, p \rangle - \innorm{x+z, \tilde{p}}\\
         &= v_{p}^\pi(s) + \inf_{\norm{x}_* \leq L_{\gamma,\beta, R_{\max}}}\langle x,p\rangle + \sup_{\tilde{p}\in\mathbb{R}^{\abs{\St}\times\abs{\A}\times\abs{\St}}}  \langle -x - z, \tilde{p}\rangle\\
        &= \inf_{\norm{x}_* \leq L_{\gamma,\beta, R_{\max}}} \sigma_{\mathbb{R}^{\abs{\St}\times\abs{\A}\times\abs{\St}}} (-x-z) + v_{p}^\pi(s) + \langle x, p\rangle\\
        &= \left\{ 
        \begin{split}
            &v_{p}^\pi(s) - \langle z, p\rangle \text{ if } \norm{z}_* \leq L_{\gamma,\beta, R_{\max}}\\
            & \infty \text{ otherwise}.
        \end{split}
        \right.
    \end{split}
\end{equation*}
Therefore, the effective domain of $(\tilde{u}_{s}^\pi)^*$ is included in the $\norm{\cdot}_*$-ball of radius $L_{\gamma,\beta, R_{\max}}$, which implies $\kappa_s^{\pi} \leq L_{\gamma,\beta, R_{\max}}$.
\end{proof}

\subsection{Proof of Theorem \ref{theorem:approximate_equivalence}}
% We proceed similarly as we did for proving Thm.~\ref{theorem:tabular_equivalence}. We first establish the following result. 

% \begin{lemma}
% \label{lemma:closure_proper_approximation}
% For all state $s\in\St$ and all policy $\pi$, define 
% the mapping $w_s^\pi: \Prec \rightarrow \mathbb{R}$ as $p \mapsto \Phi(s)^\top \w_p$. If $w_s^\pi$ is proper and closed, then its convex closure $\Breve{\text{cl}}(w_s^\pi)$ is proper. 
% \end{lemma} 

% \begin{Proof}
% We use Weierstrass theorem \cite{barbu2012convexity}[Theorem 2.8.]: since $\Prec$ is compact and $w_s^\pi$ is closed, $w_s^\pi$ takes a minimum value on $\Prec$. Moreover, we have
% $$
% \inf_{p\in\Prec}\Breve{\text{cl}}(w_s^\pi)(p) = \inf_{p\in\Prec}w_s^\pi(p)
% $$
% so since $w_s^\pi$ is proper,  $\inf_{p\in\Prec}\Breve{\text{cl}}(w_s^\pi)(p)>-\infty$ and $\Breve{\text{cl}}(w_s^\pi)$ is proper.
% \end{Proof}

% We now prove Thm.~\ref{theorem:approximate_equivalence}, whose statement is recalled below. 
\begin{theorem*}[\textbf{Linear Approximation Case}]
Let the WDRMDP $\langle \St, \A, r, \mathfrak{M}_\alpha(\hat{\mu}_n)\rangle$. Then, for any $\pi\in\Pi, s\in\St$, 
$$\inf_{\mu\in\mathfrak{M}_\alpha(\widehat{\mu}_n)}\mathbb{E}_{p\sim\mu}[\Phi(s)^\top\w_p] \geq \frac{1}{n}\sum_{i=1}^n \Phi(s)^\top \w_{\widehat{p}_i} - \eta_s^{\pi}\alpha,$$
where $\eta_s^{\pi}$ is a positive constant that depends on $s$ and $\pi$.
\end{theorem*}

\begin{proof}
The proof starts similarly as in Thm.~\ref{theorem:tabular_equivalence}[Claim (i)]. For all $i\in[n]$, define $p_i := \bigotimes_{s\in\St}p_s^i$, $\mu_i := \bigotimes_{s\in\St}\mu_s^i$, and the average distribution $\bar{\mu} := \frac{1}{n}\sum_{i=1}^n \mu_i$. Then, the worst-case distributionally robust value function can be expressed as 
\begin{equation*}
    \begin{split}
        \inf_{\mu\in\mathfrak{M}_\alpha(\widehat{\mu}_n)}\mathbb{E}_{p\sim\mu}[\Phi(s)^\top\w_p]  
        = \min_{\mu} \mathbb{E}_{p\sim \mu}[\Phi(s)^\top \w_p] \text{ s.t. } 
        \left\{
        \begin{split}
        &\mu = \bigotimes_{s\in\St} \left(\frac{1}{n}\sum_{i=1}^n \mu^i_s\right)\\
        &\frac{1}{n}\sum_{i=1}^n \mathbb{E}_{p_i\sim\mu_i} \left[\norm{p_i-\widehat{p}_i}\right] \leq \alpha,
        \end{split}
        \right.
    \end{split}
\end{equation*}
for an $\alpha$ determined by radii $\alpha_s$-s.
% For all $i\in[n]$, define $\tilde{\mu}_i := \bigotimes_{t \geq 0, s\in\St}\mu_s^i$ and introduce the notation $\bar{\mu} := \frac{1}{n}\sum_{i=1}^n \tilde{\mu}_i$.
Thus, replacing distribution $\mu$ by its constrained law and using a duality argument, we obtain 
\begin{equation*}
    \begin{split}
        \inf_{\mu\in\mathfrak{M}_\alpha(\widehat{\mu}_n)}\mathbb{E}_{p\sim  \mu}[\Phi(s)^\top \w_p]  
        = \inf_{\substack{\mu_1,\cdots,\mu_n\\ \mu_i = \bigotimes_s \mu_s^i}}\sup_{\lambda\geq 0} \left(\mathbb{E}_{p\sim \bar{\mu}}[\Phi(s)^\top \w_p]- \lambda \left(\alpha - \frac{1}{n}\sum_{i=1}^n \mathbb{E}_{p_i\sim\mu_i} \left[\norm{p_i-\widehat{p}_i}\right] \right)\right).
    \end{split}
\end{equation*}
Applying the maxmin inequality and setting all $\mu_i$-s to a Dirac distribution with full mass on the worst-case model yields
\begin{equation*}
    \begin{split}
        \inf_{\mu\in\mathfrak{M}_\alpha(\widehat{\mu}_n)} \mathbb{E}_{p\sim  \mu}[\Phi(s)^\top \w_p]  
        &\geq \sup_{\lambda\geq 0} \inf_{\substack{\mu_1,\cdots,\mu_n\\ \mu_i = \bigotimes_s \mu_s^i}} \left(\mathbb{E}_{p\sim \bar{\mu}}[\Phi(s)^\top \w_p]- \lambda \left(\alpha - \frac{1}{n}\sum_{i=1}^n\mathbb{E}_{p_i\sim\mu_i} \left[\norm{p_i-\widehat{p}_i} \right]\right)\right)\\
        &= \sup_{\lambda\geq 0} \frac{1}{n}\sum_{i=1}^n  \inf_{\substack{\mu_i: \mu_i = \bigotimes_s \mu_s^i}}\left(\mathbb{E}_{p_i\sim \mu_i}\left[\Phi(s)^\top \w_{p_i}+ \lambda \norm{p_i-\widehat{p}_i} \right]\right)-\lambda \alpha\\
        &=  \sup_{\lambda\geq 0} \frac{1}{n}\sum_{i=1}^n  \inf_{p \in  \mathcal{M}(\St)^{\abs{\St}\times\abs{\A}}}\left(\Phi(s)^\top \w_p+ \lambda\norm{p-\widehat{p}_i} \right)-\lambda \alpha.
    \end{split}
\end{equation*}
We introduce auxiliary variables $x_1,\cdots,x_n$ in order to reformulate the bound as  
\begin{equation*}
    \begin{split}
        &\inf_{\mu\in\mathfrak{M}_\alpha(\widehat{\mu}_n)} \mathbb{E}_{p\sim  \mu}[\Phi(s)^\top \w_p]  
        \geq \max_{\lambda, x_1,\cdots,x_n} \frac{1}{n}\sum_{i=1}^n x_i-\lambda \alpha \text{ s.t. } \left\{
        \begin{split}
            &\inf_{p \in  \mathcal{M}(\St)^{\abs{\St}\times\abs{\A}}}\left(\Phi(s)^\top \w_p+ \lambda
            \norm{p-\widehat{p}_i} \right) \geq x_i ,\quad \forall i\in[n] \\
            &\lambda \geq 0,
        \end{split}\right.
    \end{split}
\end{equation*}
and since $\St$ is finite, $\mathcal{M}(\St)$ is compact and the infima in the first-line constraints are minima.

For all $i\in[n]$, let the mapping $f_i: p \mapsto w_s^\pi(p) + \lambda\norm{p-\widehat{p}_i}$. By applying Prop. 1.3.17. of \cite{bertsekas2009convex}, we obtain $\Breve{\text{cl}}(f_i)(\tilde{p})= \Breve{\text{cl}}(\tilde{u}_s^\pi)(\tilde{p})+ \lambda
            \norm{\tilde{p}-\widehat{p}_i}$, where we used the fact that $\tilde{p} \mapsto \lambda\norm{\tilde{p}-\widehat{p}_i}$ is convex and closed for all $\lambda \geq 0$. 
% By Prop. 1.3.13. of \cite{bertsekas2009convex}, we thus have 
% Further consider the convex closure  $\Breve{\text{cl}}(w_s^\pi)$ of the mapping $w_s^\pi: p\mapsto \Phi(s)^\top \w_p$. Then, $f: p \mapsto \Breve{\text{cl}}(w_s^\pi)(p) + \lambda\norm{p-\widehat{p}_i}$ is also convex (\cite{bertsekas2009convex}[Prop. 1.3.17.]), and $\Breve{\text{cl}}(f) = \Breve{\text{cl}}(w_s^\pi)(p) + \lambda\norm{p-\widehat{p}_i}$. 
Thus, applying Prop.1.3.13. of \cite{bertsekas2009convex} on all $f_i$-s yields
\begin{equation*}
    \begin{split}
        % &\inf_{\mu \in\mathcal{M}_\infty} \mathbb{E}_{p\sim \mu}[V^\pi_p(s)]  \\
        &\max_{\lambda, x_1,\cdots,x_n} \frac{1}{n}\sum_{i=1}^n x_i-\lambda \alpha \text{ s.t. }
        \left\{
        \begin{split}
        &\min_{p \in \mathcal{M}(\St)^{\abs{\St}\times\abs{\A}}}\left(\Phi(s)^\top \w_p+ \lambda
            \norm{p-\widehat{p}_i} \right) \geq x_i,\quad \forall i\in [n]\\
        &\lambda \geq 0
        \end{split}
        \right.\\
    &= 
    \begin{split}
        \max_{\lambda, x_1,\cdots,x_n} \frac{1}{n}\sum_{i=1}^n x_i-\lambda \alpha \text{ s.t. }
        \left\{
        \begin{split}
        &\min_{p \in \mathcal{M}(\St)^{\abs{\St}\times\abs{\A}}}\left(\Breve{\text{cl}}(w_s^\pi)(p)+ \lambda
            \norm{p-\widehat{p}_i} \right) \geq x_i,\quad \forall i\in [n]\\
        &\lambda \geq 0.
        \end{split}
        \right.
    \end{split}
    \end{split}
\end{equation*}

Moreover, by definition of the dual norm $\norm{\cdot}_*$, the right hand side of the inequality is equivalent to the following:
\begin{equation*}
    \begin{split}
        \max_{\lambda, x_1,\cdots,x_n} \frac{1}{n}\sum_{i=1}^n x_i-\lambda \alpha \text{ s.t. }
        \left\{
        \begin{split}
        &\min_{p \in \mathcal{M}(\St)^{\abs{\St}\times\abs{\A}}}\max_{\norm{y_i}_{*}\leq \lambda} \left(\Breve{\text{cl}}(w_s^\pi)(p)+ \innorm{y_i, p-\widehat{p}_i} \right) \geq x_i ,\quad \forall i\in [n]\\
        &\lambda \geq 0,
        \end{split}
        \right.
    \end{split}
\end{equation*}
so the worst-case expectation can be reformulated as 
\begin{equation*}
    \begin{split}
        \inf_{\mu\in\mathfrak{M}_\alpha(\widehat{\mu}_n)}\mathbb{E}_{p\sim  \mu}[\Phi(s)^\top \w_p] 
        \geq 
        \begin{split}
        \max_{\lambda, x_1,\cdots,x_n} \frac{1}{n}\sum_{i=1}^n x_i-\lambda \alpha \text{ s.t. }
        \left\{
        \begin{split}
        &\min_{p \in \mathcal{M}(\St)^{\abs{\St}\times\abs{\A}}}\max_{\norm{y_i}_{*}\leq \lambda} \left(\Breve{\text{cl}}(w_s^\pi)(p)+ \innorm{y_i, p-\widehat{p}_i} \right) \geq x_i ,\quad \forall i\in [n]\\
        &\lambda \geq 0.
        \end{split}
        \right.
    \end{split}
    \end{split}
\end{equation*}

Now introduce the conjugate transform of $\Breve{\text{cl}}(w_s^\pi)$: 
\begin{equation*}
 \begin{split}
\Breve{\text{cl}}(w_s^\pi)^*(z) :=\max_{p\in \mathcal{M}(\St)^{\abs{\St}\times\abs{\A}}}\left(\langle z, p\rangle - \Breve{\text{cl}}(w_s^\pi)(p)\right).
\end{split}   
\end{equation*}
By Thm.~\ref{theorem:conjugacy_thm}(d), $\Breve{\text{cl}}(w_s^\pi)^* = (w_s^\pi)^*$. Moreover, since the feature vectors $(\Phi(s))_{s\in\St}$ are linearly independent, by \cite{bertsekas1996neuro}[Lemma 6.8.], $w_s^\pi$ is proper and closed. As a result, thanks to Lemma  \ref{prop:convex_closure_proper}, the convex closure $\Breve{\text{cl}}(w_s^\pi)$ is proper. Therefore, by Thm.~\ref{theorem:conjugacy_thm}(c), $\Breve{\text{cl}}(w_s^\pi)$ coincides with its bi-conjugate function and
\begin{equation*}
    \begin{split}
    \Breve{\text{cl}}(w_s^\pi)(p) = \Breve{\text{cl}}(w_s^\pi)^{**}(p)
    = \max_{z\in\mathfrak{W}_s^{\pi}}\left(\langle z, p\rangle - \Breve{\text{cl}}(w_s^\pi)^*(z)\right)
    = \max_{z\in\mathfrak{W}_s^{\pi}}\left(\langle z, p\rangle - (w_s^\pi)^*(z)\right),
    \end{split}
\end{equation*}
where $\mathfrak{W}_s^{\pi} := \{z:  (w_s^\pi)^*(z) <\infty\}$ is the effective domain of $(w_s^\pi)^*$.
Thus, if we use the reformulation of the convex closure and apply the minimax theorem \cite{bertsekas2009convex}[Prop. 5.5.4.] 
we obtain
\begin{equation*}
    \begin{split}
    &\min_{p \in \mathcal{M}(\St)^{\abs{\St}\times\abs{\A}}}\max_{\norm{y_i}_{*}\leq \lambda} \left(\Breve{\text{cl}}(w_s^\pi)(p)+ \innorm{y_i, p-\widehat{p}_i} \right)  \\
    &= \min_{p \in \mathcal{M}(\St)^{\abs{\St}\times\abs{\A}}} \max_{z_i\in \mathfrak{W}_s^{\pi}}\max_{\norm{y_i}_{*}\leq \lambda} -(w_s^\pi)^*(z_i) + \innorm{p, z_i}+ \innorm{y_i, p - \widehat{p}_i}\\
    &=\max_{z_i\in \mathfrak{W}_s^{\pi}} \max_{\norm{y_i}_{*}\leq \lambda} \min_{p \in \mathcal{M}(\St)^{\abs{\St}\times\abs{\A}}} -(w_s^\pi)^*(z_i) + \innorm{p, z_i}+ \innorm{y_i,  p - \widehat{p}_i}\\
    &=\max_{z_i\in \mathfrak{W}_s^{\pi}} \max_{\norm{y_i}_{*}\leq \lambda}  -(w_s^\pi)^*(z_i)- \innorm{u, \widehat{p}_i} + \min_{p \in \mathcal{M}(\St)^{\abs{\St}\times\abs{\A}}}\innorm{p, z_i + u_i} \\
    &= \max_{z_i\in \mathfrak{W}_s^{\pi}} \max_{\norm{y_i}_{*}\leq \lambda}  -(w_s^\pi)^*(z_i) - \sigma_{\mathcal{M}(\St)^{\abs{\St}\times\abs{\A}}}(-z_i - u_i) - \innorm{y_i, \widehat{p}_i},
    \end{split}
\end{equation*}
where $\sigma_{\mathcal{M}(\St)^{\abs{\St}\times\abs{\A}}}$ is the support function of ${\mathcal{M}(\St)^{\abs{\St}\times\abs{\A}}}$. 
We use the bound $\sigma_{\mathcal{M}(\St)^{\abs{\St}\times\abs{\A}}} \leq \sigma_{\mathbb{R}^{\abs{\St}\times\abs{\A}\times\abs{\St}}}$ to deduce
\begin{equation*}
    \begin{split}
    &\min_{p \in \mathcal{M}(\St)^{\abs{\St}\times\abs{\A}}}\max_{\norm{y_i}_{*}\leq \lambda} \left(\Breve{\text{cl}}(w_s^\pi)(p)+ \innorm{y_i, p-\widehat{p}_i} \right) \\
    &= \max_{z_i\in \mathfrak{W}_s^{\pi}} \max_{\norm{y_i}_{*}\leq \lambda} -(w_s^\pi)^*(z_i) - \innorm{y_i, \widehat{p}_i} - \sigma_{\mathcal{M}(\St)^{\abs{\St}\times\abs{\A}}}(-z_i- u_i) \\
    &\geq \max_{z_i\in \mathfrak{W}_s^{\pi}} \max_{\norm{y_i}_{*}\leq \lambda} -(w_s^\pi)^*(z_i) - \innorm{y_i, \widehat{p}_i} - \sigma_{\mathbb{R}^{\abs{\St}\times\abs{\A}\times\abs{\St}}}(-z_i - u_i) \\
    &= \max_{z_i\in \mathfrak{W}_s^{\pi}} \max_{\norm{z_i}_{*}\leq \lambda}
    -(w_s^\pi)^*(z_i) + \innorm{z_i, \widehat{p}_i}\\
    &= \left\{
    \begin{split}
    \Breve{\text{cl}}(w_s^\pi)(\hat{p}_i)  &\text{ if } \sup\{ \norm{z_i}_{*}: z_i \in\mathfrak{W}_s^{\pi}\} \leq \lambda \\
    -\infty &\text{ otherwise}.
    \end{split} \right.
    \end{split}
\end{equation*}
Therefore, recalling the notation $\eta_s^{\pi} := \sup\{\norm{z}_{*}: z\in\mathfrak{W}_s^{\pi}\}$, we obtain
\begin{equation*}
    \begin{split}
        &\sup_{\lambda, x_1,\cdots,x_n} \frac{1}{n}\sum_{i=1}^n x_i-\lambda \alpha \text{ s.t. } \left\{
        \begin{split}
            &\min_{p \in \mathcal{M}(\St)^{\abs{\St}\times\abs{\A}}}\max_{\norm{y_i}_{*}\leq \lambda} \left(\Breve{\text{cl}}(w_s^\pi)(p)+ \innorm{y_i, p-\widehat{p}_i} \right) \geq x_i ,\quad \forall i\in [n]\\
            &\lambda \geq 0
        \end{split}\right.\\
        &\geq \sup_{\lambda} \sup_{x_1,\cdots,x_n}\frac{1}{n}\sum_{i=1}^n x_i-\lambda \alpha \text{ s.t. } \left\{
        \begin{split}
            &\Breve{\text{cl}}(w_s^\pi)(\hat{p}_i) \geq x_i,\quad \forall i\in [n] \\
            &\lambda \geq \eta_s^{\pi}.
        \end{split}\right.
    \end{split}
\end{equation*}

% Putting this altogether yields
% \begin{equation*}
%     \begin{split}
%         \inf_{\mu\in\mathfrak{M}_\alpha(\widehat{\mu}_n)}\mathbb{E}_{p\sim  \mu}[\Phi(s)^\top \w_p] \geq  \frac{1}{n}\sum_{i=1}^n \Phi(s)^\top \w_{\hat{p}_i}  - \eta_s^{\pi} \alpha,
%     \end{split}
% \end{equation*}
% which ends the proof.

Putting this altogether yields
$ \inf_{\mu\in\mathfrak{M}_\alpha(\widehat{\mu}_n)}\mathbb{E}_{p\sim  \mu}[\Phi(s)^\top \w_p]  \geq  \frac{1}{n}\sum_{i=1}^n \Breve{\text{cl}}(w_s^\pi)(\hat{p}_i)  - \eta_s^{\pi} \alpha$. 
Now let $F: (\hat{p}_1,\cdots, \hat{p}_n)\mapsto \sum_{i=1}^n w_s^\pi(\hat{p}_i)$. By Prop.1.3.17 of \cite{bertsekas2009convex}, we have $\Breve{\text{cl}}(F)(\hat{p}_1,\cdots, \hat{p}_n) = \sum_{i=1}^n \Breve{\text{cl}}(w_s^\pi)(\hat{p}_i)$ and since 
$f\leq C$ if and only if $\Breve{\text{cl}}(f) \leq C$ for any function $f$ and constant $C$, we obtain:
$$
\inf_{\mu\in\mathfrak{M}_\alpha(\widehat{\mu}_n)}\mathbb{E}_{p\sim  \mu}[\Phi(s)^\top \w_p]   + \eta_s^{\pi} \alpha \geq \frac{1}{n}\sum_{i=1}^n w_s^\pi(\hat{p}_i),
$$
% Recalling that $\tilde{u}_s^\pi(\hat{p}_i) = v_{\hat{p}_i}^{\pi}(s)$ enables to establish 
% $v^\pi_{\mathfrak{M}_{\alpha}(\hat{\mu}_n)}(s) \geq \frac{1}{n}\sum_{i=1}^nv_{\hat{p}_i}^{\pi}(s)-\eta_s^{\pi} \alpha$, 
% $\Breve{\text{cl}}(\tilde{u}_s^\pi)(\hat{p}_i) \leq \tilde{u}_s^\pi(\hat{p}_i)$ and $\tilde{u}_s^\pi(\hat{p}_i) = v^\pi_{\hat{p}_i}(s)$.
which ends the proof.
\end{proof}

\section{Out-of-Sample Performance Guarantees}
\subsection{Proof of Theorem \ref{theorem:out_of_sample}}
Thm.~\ref{theorem:out_of_sample} is based on the following result, which is a direct consequence of \cite{fournier2015rate}[Lemma 5;Prop. 10]. 

\begin{lemma}
\label{lemma:fournier_out_of_sample_compact}
Let $\epsilon\in (0,1)$, $m:= {\abs{\St}\times\abs{\A}}$ and $\hat{\mu}_n^s\in\mathcal{M}([0,1]^m)$ be the empirical distribution at $s\in\St$. Then for all $\alpha_s\in(0,\infty)$
$$
\mu^n_s\left( \left\{\hat{p}_s :d(\hat{\mu}_n^s , \mu_s) \geq c_0\beta_s \right\}\right) \leq c_1 b_1(n_s, \beta_s) \mathbbm{1}_{\beta_s\leq 1} ,
$$
where $b_1(n_s, \beta_s) := \exp(-c_2 n_s \cdot (\beta_s)^{ m\vee 2})$ and $c_0, c_1, c_2$ are positive constants that only depend on $m\neq 2$.
\end{lemma}

The positive constant $c_0$ corresponds to the one that appears in \cite{fournier2015rate}[Lemma 5]: it bounds the degree by which the Wasserstein distance is coarser than another one defined in \cite{fournier2015rate}[Notation 4] and bounded by $1$. As a result, the Wasserstein diameter of $\mathcal{M}([0,1]^m)$ is bounded by $c_0$ so in Lemma \ref{lemma:fournier_out_of_sample_compact}, the probability vanishes when $\beta_s>1$.

We are now ready to prove Thm.~\ref{theorem:out_of_sample}. For completeness, we recall its statement below. The proof uses results from \cite{fournier2015rate} and is based on \cite{yang2018wasserstein}. There, similar guarantees were established in a stochastic control setting. Differently, our proof focuses on the MDP framework. 

\begin{theorem*}[\textbf{Finite-sample Guarantee}]
Let $\epsilon\in (0,1), m := \abs{\St}\times\abs{\A}$. Denote by $\hat{\pi}^*$ an optimal policy of the Wasserstein DRMDP $\langle \St, \A, r, \mathfrak{M}_{\alpha(n,\epsilon)}(\widehat{\mu}_n)\rangle$ and $\hat{v}^*$ its optimal value. If for all $s\in\St$ the radius of the Wasserstein ball at $s$ satisfies 
\begin{equation*}
\alpha_s(n_s,\epsilon) := \left\{
\begin{split}
    &c_0 \cdot\left(\frac{1}{n_s c_2}\log\left(\frac{c_1}{\epsilon}\right)\right)^{1/(m \vee 2)} &&\text{ if } n_s \geq C_{m}^\epsilon\\
    &c_0  && \text{ otherwise},
\end{split}
\right.
\end{equation*}
with $C_{m}^\epsilon:= \frac{1}{c_2}\log\left(\frac{c_1}{\epsilon}\right)$ and $$n_s := \sum_{i\in[n], a\in\A,s'\in\St}n_i(s,a,s'),$$
then it holds that 
\begin{align*}
 \mu^n\left(\left\{\widehat{p} \mid \E_{p\sim\mu}[v^{\hat{\pi}^*}_p(s)] \geq v^{\hat{\pi}^*}_{\mathfrak{M}_\alpha(\hat{\mu}_n)}(s), \quad  \forall s\in\St\right\}\right)\geq 1-\epsilon,
\end{align*}
where $c_0, c_1, c_2$ are positive constants that only depend on $m\neq 2$.
% \footnote{A comparable but more intricate conclusion can be established for $m=2$ \cite{fournier2015rate}[Prop. 10].}. 
\end{theorem*}

\begin{proof}
Set 
\begin{equation*}
\beta_s(n_s,\epsilon) := \left\{
\begin{split}
    & \left(\frac{1}{n_s c_2}\log\left(\frac{c_1}{\epsilon}\right)\right)^{1/(m \vee 2)} &&\text{ if } n_s \geq C_{m}^\epsilon\\
    &  && \text{ otherwise},
\end{split}
\right.
\end{equation*}
so we have $c_0\cdot \beta_s(n_s,\epsilon) = \alpha_s(n_s,\epsilon)$ and
Lemma \ref{lemma:fournier_out_of_sample_compact} ensures that the radius $\alpha_s(n_s, \epsilon)$ provides the following guarantee:
\begin{equation}
    \label{eq:guarantee}
    \mu^n_s(\{\hat{p}_s: d(\hat{\mu}_n^s , \mu_s) \leq \alpha_s(n_s, \epsilon) \}) \geq 1 - \epsilon.
\end{equation}
Now, we introduce the following operators 
% $T^{\hat{\pi}^*_s}_{\mu_s}$ where by assumption, $\hat{\pi}^* = (\hat{\pi}^*_s)_{s\in\St}$ is the optimal policy of the Wasserstein DRMDP $\langle \St, \A, r, \mathfrak{M}_{\alpha(n,\epsilon)}(\widehat{\mu}_n)\rangle$:
\begin{equation*}
    \begin{split}
    T_{\mathfrak{M}_\alpha(\hat{\mu})}v(s) := &\sup_{\pi_s\in\A}\inf_{\mu_s\in\mathfrak{M}_{\alpha_s}(\widehat{\mu}_s)} T_{\mu_s}^{\pi_s}v(s) 
    ,\quad  \forall s\in\St,
    \end{split}
\end{equation*}
where 
\begin{align*}
    T_{\mu_s}^{\pi_s}v(s) :&=  r(s,\pi_s)+\gamma \int_{p_s\in\Prec_s}\sum_{s'\in\St} v(s')p_{s}(s'|s,\pi_s)d\mu_s(p_s),
\end{align*}
and consider $\hat{\pi}^* = (\hat{\pi}^*_s)_{s\in\St}$ an optimal policy of the Wasserstein DRMDP $\langle \St, \A, r, \mathfrak{M}_{\alpha(n,\epsilon)}(\widehat{\mu}_n)\rangle$\footnote{Although this work is restricted to the set of deterministic policies, we do not lose generality as long as $\Prec$ is $(s,a)$-rectangular because then, one can find an optimal policy that is stationary deterministic \cite{wiesemann2013robust}.}. By Equation~\eqref{eq:guarantee}, we have the following one-step guarantee:
$$
\mu^n_s(\{\hat{p}_s: T^{\hat{\pi}^*_s}_{\mu_s}v(s) \geq T_{\mathfrak{M}_{\alpha}(\hat{\mu}_n)}v(s) \}) \geq 1-\epsilon.
$$

Remarking that $T^{\hat{\pi}^*_s}_{\mu_s}$ is a non-decreasing $\gamma$-contraction w.r.t. the sup-norm (see \cite{chen2019distributionally}[Lemma 4.2.]), we show by induction on $k\geq 1$ that if $\mu_s \in \mathfrak{M}_{\alpha_s(n_s, \epsilon)}(\hat{\mu}_n^s)$, then $(T^{\hat{\pi}^*_s}_{\mu_s})^kv(s)\geq (T_{\mathfrak{M}_{\alpha}(\hat{\mu}_n)})^k v(s)$. By definition of $T_{\mathfrak{M}_{\alpha}(\hat{\mu}_n)}$, we have $T^{\hat{\pi}^*_s}_{\mu_s}v(s) \geq T_{\mathfrak{M}_{\alpha}(\hat{\mu}_n)}v(s)$. Supposing that the condition holds for an arbitrary $k\geq 1$, we have 
$$
(T^{\hat{\pi}^*_s}_{\mu_s})^{k+1}v(s) = T^{\hat{\pi}^*_s}_{\mu_s}((T^{\hat{\pi}^*_s}_{\mu_s})^k v)(s)\geq T^{\hat{\pi}^*_s}_{\mu_s}((T_{\mathfrak{M}_{\alpha}(\hat{\mu}_n)})^k) v(s)\geq 
T_{\mathfrak{M}_{\alpha}(\hat{\mu})}(T_{\mathfrak{M}_{\alpha}(\hat{\mu}_n)})^k v(s),
$$
so the induction assumption holds for all $k\geq 1$. Using the contracting property of $T^{\hat{\pi}^*_s}_{\mu_s}$, we have $\lim_{k\rightarrow\infty}(T^{\hat{\pi}^*_s}_{\mu_s})^kv(s) = \E_{p\sim\mu}[v^{\hat{\pi}^*}_{p}(s)]$ and by applying \cite{chen2019distributionally}[Thm. 4.5.], we obtain $\lim_{k\rightarrow\infty}(T_{\mathfrak{M}_{\alpha}(\hat{\mu})})^kv(s) = v^{\hat{\pi}^*}_{\mathfrak{M}_\alpha(\hat{\mu}_n)}(s)$. Therefore, setting $k\rightarrow\infty$, if $\mu_s \in \mathfrak{M}_{\alpha_s(n_s, \epsilon)}(\hat{\mu}_n^s)$, then $\E_{p\sim\mu}[v^{\hat{\pi}}_{p}(s)] \geq v^{\hat{\pi}^*}_{\mathfrak{M}_\alpha(\hat{\mu}_n)}(s)$ and the following probabilistic guarantee holds for all $s\in\St$:
$$
\mu^n_s(\{\hat{p}_s: \E_{p\sim\mu}[v^{\hat{\pi}}_{p}(s)] \geq v^{\hat{\pi}^*}_{\mathfrak{M}_\alpha(\hat{\mu}_n)}(s) \}) \geq 1- \epsilon,
$$
which can be rewritten as 
$$
\mu^n_s(\{\hat{p}_s: \E_{p\sim\mu}[v^{\hat{\pi}}_{p}(s)] < v^{\hat{\pi}^*}_{\mathfrak{M}_\alpha(\hat{\mu}_n)}(s) \}) \leq \epsilon.
$$
The independence structure $\mu^n = \bigotimes_{s\in\St}\mu^n_s$ enables to obtain
$$
\mu^n(\{\hat{p}: \E_{p\sim\mu}[v^{\hat{\pi}}_{p}(s)] < v^{\hat{\pi}^*}_{\mathfrak{M}_\alpha(\hat{\mu}_n)}(s), \quad \forall s\in\St\}) \leq \prod_{s\in\St} \epsilon \leq \epsilon,
$$
which concludes the proof, by taking the complementary event. 
\end{proof}

\end{document}